\documentclass[a4paper]{article}
\usepackage[utf8]{inputenc}
\usepackage[T1]{fontenc}
\usepackage{amsmath}
\usepackage{amssymb}
\usepackage{bm}
\usepackage{mathrsfs}
\usepackage{amsthm}
\usepackage{thmtools,thm-restate}
\usepackage{stmaryrd}
\usepackage{fullpage}
\usepackage[margin=1in]{geometry}
\usepackage{setspace}
\usepackage{soul}
\usepackage{graphicx}
\usepackage{dsfont}
\usepackage{color}
\usepackage{listings}
\usepackage{xcolor}
\usepackage{float}
\usepackage{mathtools}
\usepackage{enumitem}
\usepackage{etoolbox}
\usepackage[ruled]{algorithm2e}

\makeatletter
\newenvironment{alg:equation}
{\begin{equation}}
{\@endalgocfline\vspace{-\baselineskip}\end{equation}\hfill\strut\par}
\makeatother
\makeatletter
\newenvironment{alg:equation*}
{\begin{equation*}}
{\@endalgocfline\vspace{-\baselineskip}\end{equation*}\hfill\strut\par}
\makeatother
\makeatletter
\newenvironment{alg:problem}
{\begin{problem}}
{\@endalgocfline\vspace{-\baselineskip}\end{problem}\hfill\strut\par}
\makeatother

\usepackage{algpseudocode}
\usepackage{subcaption}
\usepackage[pdftex, pdfborder={0 0 0}, colorlinks=true, linkcolor=blue, citecolor=red, pagebackref=false, hypertexnames=false]{hyperref}
\usepackage[compress]{cleveref}

\declaretheorem[within=section,name=Proposition,refname={Proposition,Propositions},Refname={Proposition,Propositions}]{prop}
\declaretheorem[sibling=prop,name=Lemma,refname={lemma,lemmas},Refname={Lemma,Lemmas}]{lemma}
\declaretheorem[sibling=prop,name=Theorem,refname={theorem,theorems},Refname={Theorem,Theorems}]{theorem}

\declaretheorem[sibling=prop,name=Remark,refname={remark,remarks},Refname={Remark,Remarks},style=remark]{remark}
\declaretheorem[sibling=prop,name=Corollary,refname={corollary,corollaries},Refname={Corollary,Corollaries}]{corollary}

\newcommand{\converge}[4]{#1\underset{#3\to #4}{\longrightarrow}#2}
\newcommand{\argmin}[1]{\underset{#1}{\text{\upshape{argmin }}}}
\newcommand{\argmax}[1]{\underset{#1}{\text{\upshape{argmax }}}}
\newcommand{\minimize}[1]{\underset{#1}{\text{\upshape{minimize }}}}

\newcommand{\oracle}{\text{\upshape{oracle}}}

\newcommand{\step}[1]{%
\par
\addvspace{\medskipamount}%
\noindent\textbf{Step~#1.}%
}
\makeatother

\def\E{\mathcal E}
\def\M{\mathcal M}
\def\P{\mathcal P}
\def\R{\mathbb R}
\def\N{\mathbb N}
\def\X{\mathcal X}

\DeclarePairedDelimiterXPP{\descrnorm}[3]{}{\lVert}{\rVert}
{
\ifblank{#2}{}{_{#2}} 
\ifblank{#3}{}{^{#3}}
}{
\ifblank{#1}{\:\cdot\:}{#1}
}
\newcommand{\norm}[1]{\descrnorm*{#1}{}{}}
\newcommand{\descrsqnorm}[2]{\descrnorm*{#1}{#2}2}
\newcommand{\sqnorm}[1]{\descrsqnorm{#1}{}}
\newcommand{\descrEnorm}[3]{\descrnorm*{#1}{\ifblank{#2}{\E}{\E_{#2}}}{#3}}
\newcommand{\Enorm}[2]{\descrEnorm{#1}{#2}{}}
\newcommand{\sqEnorm}[2]{\descrEnorm{#1}{#2}2}
\newcommand{\descrXnorm}[3]{\descrnorm*{#1}{\ifblank{#2}{\X}{\X_{#2}}}{#3}}
\newcommand{\Xnorm}[2]{\descrXnorm{#1}{#2}{}}

\DeclarePairedDelimiterXPP{\descrInnerProd}[3]{}{\langle}{\rangle}
{
\ifblank{#3}{}{_{#3}}
}
{
\ifblank{#1}{\:\cdot\:}{#1},\mathopen{}
\ifblank{#2}{\cdot\:}{#2}
}
\newcommand{\InnerProd}[2]{\descrInnerProd*{#1}{#2}{}}
\newcommand{\EInnerProd}[3]{\descrInnerProd*{#1}{#2}{\ifblank{#3}{\E}{\E_{#3}}}}

\DeclarePairedDelimiterXPP\descrdist[3]{d}{\lparen}{\rparen}{^{#3}}{
\ifblank{#1}{\:\cdot\:}{#1},\mathopen{}
\ifblank{#2}{\cdot\:}{#2}
}
\newcommand{\dist}[2]{\descrdist*{#1}{#2}{}}
\newcommand{\sqdist}[2]{\descrdist*{#1}{#2}2}

\DeclarePairedDelimiterXPP{\textpar}[2]{\text{\upshape{#1}}}{\lparen}{\rparen}{}{#2}
\newcommand{\conv}[1]{\textpar*{conv}{#1}}
\newcommand{\inter}[1]{\textpar*{int}{#1}}
\newcommand{\dom}[1]{\textpar*{dom}{#1}}
\newcommand{\epi}[1]{\textpar*{epi}{#1}}

\DeclarePairedDelimiterXPP{\textacc}[2]{\text{\upshape{#1}}}{\{}{\}}{}{#2}
\newcommand{\conva}[1]{\textacc*{conv}{#1}}

\DeclarePairedDelimiterXPP{\symbpar}[3]{\ifblank{#1}{}{#1}}{\lparen}{\rparen}{\ifblank{#2}{}{^#2}}{#3}
\newcommand{\cloconv}[1]{\symbpar*{\overline{\text{\upshape{conv}}}}{}{#1}}
\newcommand{\prob}[1]{\symbpar*{\P}{}{#1}}

\newcommand{\support}[2]{\ifblank{#2}{\sigma_{#1}}{\symbpar*{\sigma_{#1}}{}{#2}}}
\newcommand{\indic}[2]{\ifblank{#2}{\iota_{#1}}{\symbpar*{\iota_{#1}}{}{#2}}}
\newcommand{\proj}[2]{\ifblank{#1}{\text{\upshape{proj}}_{#2}}{\symbpar*{\text{\upshape{proj}}_{#2}}{}{#1}}}
\newcommand{\lscc}[1]{\symbpar*{\Gamma}{}{#1}}
\newcommand{\plscc}[1]{\symbpar*{\Gamma_0}{}{#1}}
\newcommand{\grad}[2]{\ifblank{#2}{\nabla #1}{\symbpar*{\nabla #1}{}{#2}}}
\newcommand{\subgrad}[2]{\ifblank{#2}{\partial #1}{\symbpar*{\partial #1}{}{#2}}}
\newcommand{\descrC}[4]{\symbpar*{\mathcal C\ifblank{#1}{}{_{#1}}}{#4}{#2\ifblank{#3}{}{\mathopen{};#3}}}
\newcommand{\descrM}[2]{\symbpar*{\M\ifblank{#1}{}{^{#1}}}{}{#2}}
\newcommand{\Fen}[2]{\ifblank{#2}{{#1}^\ast}{\symbpar*{{#1}^*}{}{#2}}}

\DeclarePairedDelimiterXPP{\symbacc}[2]{{#1}}{\{}{\}}{}{#2}

\DeclarePairedDelimiterX{\abs}[1]{\lvert}{\rvert}{#1}
\DeclarePairedDelimiterXPP{\sqabs}[1]{}{\lvert}{\rvert}{^2}{#1}

\DeclarePairedDelimiterX{\IntInt}[2]{\{}{\}}{#1,\,\ldots,#2}
\DeclarePairedDelimiterXPP{\descrBreg}[4]{B_{#1}}{\lparen}{\rparen}{}
{
\ifblank{#2}{\:\cdot\:}{#2},\mathopen{}
\left(\ifblank{#3}{\cdot\:}{#3},\mathopen{}\ifblank{#4}{\cdot\:}{#4}\right)
}

\newcommand{\xiBreg}[3]{\descrBreg*{\Xi}{#1}{#2}{#3}}
\newcommand{\fstarBreg}[3]{\descrBreg*{{\Fen f{}}}{#1}{#2}{#3}}

\newenvironment{problem}[2]
{
\begin{equation}
\ifblank{#1}{}{\tag{#1}}
\ifblank{#2}{}{\label[problem]{#2}}
}
{\end{equation}\ignorespacesafterend}
\crefformat{problem}{problem~(#2#1#3)}
\labelcrefformat{problem}{#2#1#3}
\crefrangeformat{problem}{problems~(#3#1#4) to~(#5#2#6)}
\crefmultiformat{problem}{problems~(#2#1#3)}%
{ and~(#2#1#3)}{, (#2#1#3)}{, and~(#2#1#3)}
\crefname{problem}{problem}{problems}

\newenvironment{alproblem}[2]
{
\align
\ifblank{#1}{}{\tag{#1}}
\ifblank{#2}{}{\label[problem]{#2}}
}
{\endalign\ignorespacesafterend}
\crefformat{alproblem}{problem~(#2#1#3)}
\creflabelformat{alproblem}{(#2#1#3)}
\crefrangeformat{alproblem}{problems~(#3#1#4) to~(#5#2#6)}
\crefmultiformat{alproblem}{problems~(#2#1#3)}%
{ and~(#2#1#3)}{, (#2#1#3)}{, and~(#2#1#3)}
\crefname{alproblem}{problem}{problems}

\crefformat{assum}{assumption~\textnormal(#2#1#3\textnormal)}
\crefrangeformat{assum}{assumptions~\textnormal(#3#1#4\textnormal) to~\textnormal(#5#2#6\textnormal)}
\crefmultiformat{assum}{assumptions~\textnormal(#2#1#3\textnormal)}%
{ and~\textnormal(#2#1#3\textnormal)}{, \textnormal(#2#1#3\textnormal)}{, and~\textnormal(#2#1#3\textnormal)}
\crefname{assum}{assumption}{assumptions}

\crefformat{subres}{#2\textnormal(#1\textnormal)#3}
\creflabelformat{subres}{\textnormal(#2#1#3\textnormal)}

\newcounter{algo}

\crefformat{algo}{algorithm~#2#1#3}
\crefrangeformat{algo}{algorithms~#3#1#4 to~#5#2#6}
\crefmultiformat{algo}{algorithms~#2#1#3}%
{ and~#2#1#3}{, #2#1#3}{, and~#2#1#3}
\crefname{algo}{algorithm}{algorithms}

\crefname{line}{line}{lines}

\title{A nonsmooth Frank--Wolfe algorithm through a dual cutting-plane approach}
\author{Guilherme Mazanti\thanks{Université Paris-Saclay, CNRS, CentraleSupélec, Inria, Laboratoire des signaux et systèmes, 91190, Gif-sur-Yvette, France.}{ }\thanks{Fédération de Mathématiques de CentraleSupélec, 91190, Gif-sur-Yvette, France.} \and Thibault Moquet\footnotemark[1]{ }\footnotemark[2] \and Laurent Pfeiffer\footnotemark[1]{ }\footnotemark[2]}
\date{\today}

\begin{document}

\setlist[itemize, 1]{itemsep=0pt}

\maketitle

\begin{abstract}
An extension of the Frank--Wolfe Algorithm (FWA), also known as Conditional Gradient algorithm, is proposed. In its standard form, the FWA allows to solve constrained optimization problems involving $\beta$-smooth cost functions, calling at each iteration a Linear Minimization Oracle. More specifically, the oracle solves a problem obtained by linearization of the original cost function. The algorithm designed and investigated in this article, named Dualized Level-Set (DLS) algorithm, extends the FWA and allows to address a class of nonsmooth costs, involving in particular support functions. The key idea behind the construction of the DLS method is a general interpretation of the FWA as a cutting-plane algorithm, from the dual point of view. The DLS algorithm essentially results from a dualization of a specific cutting-plane algorithm, based on projections on some level sets. The DLS algorithm generates a sequence of primal-dual candidates, and we prove that the corresponding primal-dual gap converges with a rate of $O(1/\sqrt{t})$.
\end{abstract}

\medskip


\medskip


\section{Introduction}

The Frank--Wolfe Algorithm (FWA), also known as Conditional Gradient Algorithm, is an iterative minimization algorithm which was first introduced in \cite{FW}. It aims at solving numerically problems of the form
\begin{problem}{}{pb:main_pb}
\minimize{x\in K}f(x),
\end{problem}
where $f$ is a convex function with Lipschitz-continuous gradient and $K$ is a closed convex bounded set of some Banach space $\X$. This method relies on a Linear Minimization Oracle (LMO) of the form
\begin{equation}\tag{LMO$_\mu$} \label{eq:LMO_mu}
\minimize{x\in K}\InnerProd\mu x,
\end{equation}
for some well-chosen $\mu\in\X^*$. The result of this oracle is used to update the candidate to optimality through convex combinations.

One simple choice consists in taking, at iteration $t$,
\begin{equation*}
\mu^t=\grad f{x^t},\quad \gamma_t = 2/(t+2), \quad\text{ and }\quad x^{t+1} = \gamma_t v^t+(1-\gamma_t)x^t,
\end{equation*}
where $v^t$ is a solution to \labelcref{eq:LMO_mu} with $\mu= \mu^t$. We will refer to that method as \textit{agnostic FWA}. It is well-known, (see for instance \cite{Jaggi}) that the agnostic FWA converges (in value) to a minimizer of $f$ over $K$ with a speed of order $1/t$.

There are many other possible choices for iteration updates. We mention the FWA with line-search and the fully-corrective FWA, which consist in taking
\begin{problem}{}{pb:fcfw}
x^{t+1} \in \argmin{x\in \tilde K^t} f(x),
\end{problem}
with $\tilde K^t$ respectively defined by $\conva{x^t,v^t}$ and $\conva{v^0,\ldots,v^t}$. In other words, we replace in \Cref{pb:main_pb} the feasible set $K$ by an inner polyhedric approximation, the set $\tilde K^t$. These variants enjoy in general the same sublinear convergence properties; yet improved rates of convergence can be obtained in many situations, see \cite{wirth2023acceleration}.

A great number of applications of the FWA can be found in \cite{Jaggi,pokutta2023frank} and references therein. We mention, among others, machine learning (see \cite{Jaggi,lacoste2013block}), optimal transport (see \cite{flamary2016optimal}), image processing (see \cite{joulin2014efficient}), and potential mean-field games (see \cite{lavigne2023generalized,liu2023mean}).

The article is dedicated to the design of an extension of the FWA which can handle nonsmooth cost functions, and which we call Dualized Level-Set (DLS) method. It allows to solve problems of the form
\begin{equation*}
\minimize{x\in\E_1}f(x)+\support Q{Ax-b}+\indic Kx,
\end{equation*}
where $\E_1$ and $\E_2$ are Hilbert spaces, $A\colon\E_1\to\E_2$ is a bounded linear operator, $b \in \E_2$, $f$ is convex and has Lipschitz-continuous gradient, $K \subset \E_1$ is a set for which we have an LMO, and $Q \subset \E_2$ is of the form $Q=Q_1+Q_2$, where $Q_1$ is a closed convex bounded set and $Q_2$ is a closed convex cone. The term $\support Q{Ax-b}$ is in general non-differentiable with respect to $x$. It typically describes equality constraints of the form $Ax = b$ (if $Q = \E_2$) or finitely many inequality constraints of the form $Ax \leq b$ (if $Q$ is the closed positive orthant of $\E_2 = \R^m$).

There already exist some extensions of the FWA to a nonsmooth framework. We mention the Frank-Wolfe Augmented Lagrangian (FW-AL) method from \cite{gidel2018frank}, the Con\-di\-tio\-nal-Gradient-based Augmented Lagrangian (CGAL) algorithm from \cite{yurtsever2019conditional} and the Conditional Gradient with Augmented Lagrangian and Proximal-step (CGALP) from \cite{Silveti_Falls_2020}. All rely on a Moreau regularization of the nonsmooth term ($\sigma_Q$ in our framework), which amounts to the augmented Lagrangian method when $\sigma_Q(Ax-b)$ models inequality or equality constraints. It is proved in \cite[Theorems~4.1 and 4.2]{Silveti_Falls_2020} that the sequence of candidates generated by the CGALP method is asymptotically feasible and that the associated Lagrangian values converge at a rate of $o(1/t^{b})$, where $b \in [0, 1/3)$ is a parameter of the algorithm (see also the discussion provided in \cite[Examples~3.4 and 4.4]{Silveti_Falls_2020}).

We follow a different approach, based on an interpretation of the fully-corrective FWA as a cutting-plane algorithm, from the point of view of the dual problem. This interpretation is not new and can be found in \cite[Section~7.7]{MAL-039} and in \cite{zhou2018limited}. In \cite{zhou2018limited}, the authors use a specific implementation of the Frank--Wolfe algorithm to construct a new variant of cutting-plane algorithm that enjoys a linear convergence. As we will explain below, we follow the reverse path. Let us mention that \cite{Bach2015Duality} gives a dual interpretation of the agnostic FWA as a mirror descent algorithm. We will explain more in detail the dual interpretation of the fully-corrective FWA in \Cref{sec:extension}, we only give a rough description of it in this introduction. Observe that the dual of \Cref{pb:main_pb} is given by 
\begin{equation*}
\minimize{\mu \in \X^*} \Fen f\mu + \sigma_K(-\mu).
\end{equation*}
The dual of \Cref{pb:fcfw} is the same problem, with $K$ replaced by $\tilde{K}^t$. Since $\tilde K^t \subset K$, $\sigma_{\tilde K^t}$ is a lower approximation of $\sigma_K$; moreover, since $\tilde K^t$ is polyhedral, $\sigma_{\tilde K^t}$ is piecewise affine. So the FWA algorithm (with line search or the fully corrective variant) amounts to solving at each iteration a simplified version of the dual problem obtained through a piecewise-affine lower approximation of the dual cost: this is the basic principle of every cutting-plane-type method. To be rigorous, let us note that cutting-plane methods usually approximate the whole cost function while here, only the second term is approximated. Such methods are referred to as simplicial methods in \cite{MAL-039,zhou2018limited}. We find it convenient to keep the terminology cutting-plane in this article.

In a nutshell, our general strategy for the design of the desired extension of the FWA consists in ``dualizing'' a cutting-plane-type algorithm. This strategy brings two main difficulties. First, we need convergence guaranties for the dualized algorithm (and not just for the chosen cutting-plane-type algorithm). Second, the dualized algorithm must be implementable.
Our attention has focused on the method introduced in \cite[Section~2.2.1]{lemarechal1995new}, which we will refer to as the Level-Set method. While the simplest cutting-plane algorithms simply consists in minimizing a piecewise affine approximation of the cost function, the Level-Set method updates the current candidate to optimality by projecting it onto a level set of the piecewise affine approximation. The addition of this projection step attenuates the instabilities from which the basic cutting-plane methods suffer. Moreover, it enables the authors of \cite{lemarechal1995new} to perform a quantitative convergence analysis. They indeed prove that the Level-Set method exhibits a rate of convergence of $O(1/\sqrt t)$. This rate of convergence is actually established for some quantity denoted $\Delta(t)$, which we interpret as a primal-dual gap. The fact that not only the optimality gap, but also the primal-dual gap, converges (with a certain rate) is a crucial aspect, since it allows to show the convergence of FWA-type algorithm obtained through dualization. The convergence of the primal-dual gap in the Level-Set method is our main interest for it.

Our DLS method is not a direct dualization of the Level-Set method, but rather a dualization of an extension---in two directions---of the Level-Set method.
In the original formulation of the Level-Set method, the full cost function is approximated with a piecewise affine cost. In our algorithm, the term $f^*$ of the dual problem remains unchanged, which requires us to proceed to an extension of the convergence proof of the Level-Set method in which only some part of the cost function is approximated (through a piecewise affine function). The second extension of the Level-Set method that we need to perform concerns the projection step. In the original method, the projection step is done with respect to the Euclidean norm. The dualization of this step would require the knowledge of $f^*$, which we consider as too demanding. We propose to change the Euclidean norm by a specific Bregman distance, which ultimately yields a more tractable projection step. The DLS algorithm enjoys the same convergence rate as the original Level-Set method, in $O(1/\sqrt t)$. Let us stress however that this convergence rate does not account for the possible increase of complexity of each iteration.

This article is organized as follows. In \Cref{sec:prelim} we present our notations and some preliminary results. In \Cref{sec:extension} we give an insight on the primal and dual interpretation of the Frank--Wolfe Algorithm, and we then present our Extended Level-Set (ELS) method and its theoretical guarantees. In \Cref{sec:DLS} we derive our DLS method, obtained by application of the ELS method to the dual problem, and we prove its convergence. The proofs of the technical results are postponed to \Cref{sec:proofs}. Finally, \Cref{sec:numerics} is dedicated to numerical examples.

\section{Preliminaries and notations}\label{sec:prelim}

\paragraph{General notations}

In the following, $\bar\R$ denotes the ordered set $\R\cup\{+\infty,-\infty\}$ and $\bar\R^+$ the ordered set $\R\cup\{+\infty\}$. Let $\E$ be a Hilbert space. Its inner product is denoted as $\EInnerProd{}{}{}$ and the deriving norm $\Enorm{}{}$, defined for all $x\in\E$ as $\Enorm{x}{}=\sqrt{\EInnerProd xx{}}$. When we deem that no confusion is possible, we will drop the subscript for the inner product and the norm. Let $\X$ be a Banach space endowed with the norm $\Xnorm{}{}$. We denote as $\X^*$ its topological dual, i.e., the space of continuous linear forms over $\X$, as $\descrnorm{}{\X^*}{}$ the associated dual norm, and as $\descrInnerProd{}{}{\X^*,\X}$ the natural pairing. Likewise, we will drop the subscripts when we deem that no confusion is possible. We recall that $\X^*$ endowed with $\descrnorm{}{\X^*}{}$ is a Banach space and that a Hilbert space $\E$ is a Banach space. Also, we always identify $\E^*$ with $\E$. In this section, we make the convention that spaces denoted by the letter $\E$ (possibly with subscripts) are always assumed to be Hilbert spaces, while spaces denoted by the letter $\X$ (possibly with subscripts) are always assumed to be Banach spaces.

Unless specified otherwise when required, the definitions below are taken similarly over both $\X$ and $\X^*$. Let $f\colon\X\to\bar\R$. We denote as $\epi f$ the epigraph of $f$, defined as
\begin{equation*}
\epi f = \left\{(x,\lambda)\in\X\times\R\,\big|\,\lambda\geq f(x)\right\}.
\end{equation*}
We recall that $f$ is convex (respectively (weakly) lower semicontinuous) \textit{iff} $\epi f$ is convex (respectively (weakly) closed).
We denote as $\lscc\X$ the set of convex lower semicontinuous functions from $\X$ to $\bar\R$ and as $\plscc\X$ the subset of those which are proper, i.e., which never take the value $-\infty$ and are not constant equal to $+\infty$. 

In what follows, we assume $f\in\plscc\X$. We denote as $\dom f$ the domain of $f$, which is the set
\begin{equation*}
\dom f = \left\{f\in\R\right\} = \left\{x\in\R\,\big|\,f(x)\in\R\right\}.
\end{equation*}
For any $x\in\X$, we denote as $\subgrad fx$ the set of its subgradients at $x$, i.e., the set
\begin{equation*}
\subgrad fx=\left\{\mu\in\X^*\,\big|\,\forall y\in\X,\, f(y)\geq f(x)+\InnerProd\mu{y-x}\right\},
\end{equation*}
and as $\dom{\subgrad f{}}$ the set
$\dom{\subgrad f{}}=\left\{\subgrad f{}\neq\varnothing\right\}$. Notice that $\dom{\partial f}\subset\dom f$. For a function $g\in\plscc{\X^*}$ and $\mu\in\X^*$, we define $\subgrad g\mu$ as
\begin{equation*}
\subgrad g\mu = \left\{x\in\X\,\big|\,\forall \lambda\in\X^*,\,g(\lambda)\geq g(\mu)+\InnerProd{\lambda-\mu}x\right\}.
\end{equation*}

Let $E\subset\X$, $\bar x\in\X$, and $F$ be a subset of some topological space. We denote as
\begin{itemize}
\item $\conv E$ the convex hull of $E$, and $\cloconv E$ the set $\overline{\conv E}$. We omit the parentheses when $E$ is given as the description of its elements;
\item $\indic E{}\colon\X\to\{0,+\infty\}$ the characteristic function of $E$, defined for all $x\in\X$ as
\begin{equation*}
\indic Ex=
\begin{dcases*}
0 & if $x\in E$, \\
+\infty & otherwise;
\end{dcases*}
\end{equation*}
\item $\support E{}\colon\X^*\to\bar\R$ the support function of $E$, defined for all $\mu\in\X^*$ as 
\begin{equation*}
\support E\mu=\sup_{x\in E}\InnerProd\mu x.
\end{equation*}
We recall that we have $\sigma_E=\sigma_{\cloconv E}$.
\end{itemize}
We now assume that $E$ is nonempty. We denote as
\begin{itemize}
\item $N_E(\bar x)$, the normal cone of $E$ at $\bar x$, defined as $\subgrad{\indic E{}}{\bar x}$, or in explicit terms as
\begin{equation*}
N_E(\bar x)=
\begin{dcases*}
\varnothing & if $\bar x \notin E$,\\
\left\{\mu\in\X^* \,\big|\,\forall x \in E,\, \InnerProd\mu{x-\bar x} \leq 0 \right\}& otherwise;
\end{dcases*}
\end{equation*}
\item $\dist{}E\colon\X\to\R$ the distance to $E$, defined for all $x\in\X$ as
\begin{equation*}
\dist xE=\inf_{x'\in E}\norm{x-x'};
\end{equation*}
\item $\descrC{}EF{}$ the set of continuous functions over $E$ taking values in $F$, as $\descrC{}E{}{}$ the space $\descrC{}E\R{}$, and as $\descrC bE{}{}$ the subset of $\descrC{}E{}{}$ of bounded functions over $E$;
\item $\descrM{}E$ the space of signed Radon measures over $E$, $\descrM+E$ the subset of $\descrM{}E$ of nonnegative measures, and $\prob E$ the set of probability measures over $E$, i.e., the subset of $\descrM+E$ of measures $m$ of total mass $m(E)=1$.
\end{itemize}

\paragraph{Duality}

We denote the Legendre--Fenchel transform, or conjugate, of $f$ as $\Fen f{}\colon\X^*\to\bar\R$. It is defined for all $\mu\in\X^*$ as
\begin{equation*}
\Fen f\mu=\sup_{x\in\X}\InnerProd\mu x-f(x).
\end{equation*}
Notice that $f^*$ lies in $\lscc{\X^*}$, as a supremum of convex (lower semi-)continuous functions of $\mu$.
For $g\colon\X^*\to\bar\R$, its conjugate is the function $\Fen g{}\colon\X\to\bar\R$ defined for all $x\in\X$ as
\begin{equation*}
\Fen gx=\sup_{\mu\in\X^*}\InnerProd\mu x-g(\mu),
\end{equation*}
and we denote as $f^{**}=\left(f^*\right)^*$ the biconjugate of $f$. Notice that, with this definition, for any $E\subset\X$, we have $\support E{}=\Fen{\indic E{}}{}$. It follows from the definition of the Legendre--Fenchel transform that, for any $f\colon\X\to\bar\R$, $g\colon\X^*\to\bar\R$, and $(x, \mu) \in \X\times\X^*$,
\begin{equation} \label{eq:fenchel_young}
f(x)+\Fen f\mu\geq\InnerProd\mu x\quad\text{and}\quad g(\mu)+\Fen gx\geq\InnerProd \mu x,
\end{equation}
these inequalities being known as the Fenchel--Young inequality. We also recall the Fenchel--Moreau Theorem \cite[Theorem~2.3.3]{zualinescu}.

\begin{theorem}[name = Fenchel--Moreau,label = th:FenchelMoreau]
Assume that $f\colon\X\to\bar\R^+$ and $f\not\equiv+\infty$. Then $f\in\plscc\X$ \textit{iff} $f^{**}=f$, and in that case $f^*\in\plscc{\X^*}$.
\end{theorem}

The following result is an easy consequence of \Cref{th:FenchelMoreau} and the definitions of $f^*$, $\partial f$, and $\partial f^*$.

\begin{lemma}[label=lem:FenchelYoungEq]
Let $f\in\plscc\X$ and $(x,\mu)\in\X\times\X^*$. Then
\begin{equation*}
\mu \in \subgrad fx
\Leftrightarrow
f(x)+\Fen f\mu=\InnerProd\mu x
\Leftrightarrow
x\in\subgrad{\Fen f{}}\mu.
\end{equation*}
\end{lemma}

\begin{remark}[label=rem:subgrad_support]
Let $E\subset\X$ be a nonempty closed convex set and $\mu\in\X^*$. From \Cref{lem:FenchelYoungEq}, we have
\begin{equation*}
\subgrad{\support E{}}{\mu}=\argmax{v\in E}\InnerProd\mu v=\left\{v\in E \,\big|\,\support E\mu=\InnerProd\mu v\right\}.
\end{equation*}
\end{remark}

Using \Cref{rem:subgrad_support} and the fact that $\displaystyle \inf_{v\in K}\InnerProd\mu v = -\sup_{v\in K}\InnerProd\mu{-v}$, we obtain at once the following result.

\begin{corollary}[label = coro:subgrad_oracle]
Let $K\subset\X$ be a nonempty closed convex set and $\mu \in \X^*$. Then 
\begin{equation*}
\argmin{v\in K}\InnerProd\mu v = -\subgrad{\support{-K}{}}\mu
\end{equation*}
\end{corollary}

We next recall the Fenchel--Rockafellar duality theorem (see \cite[Theorem~2.8.3 and Corollary 2.8.5]{zualinescu}), which plays a major role in this work.

\begin{theorem}[name= Fenchel--Rockafellar, label=th:FenchelRockafellar]
Let $\X_1$ and $\X_2$ be two Banach spaces, $f \in \plscc{\X_1}$, $g \in \plscc{\X_2}$, and $A \colon \X_1 \rightarrow \X_2$ be a bounded linear operator. Assume that
\begin{equation} \label{eq:qualif_cond_gal}
0 \in \inter{A\, \dom f - \dom g}.
\end{equation}
Then, the following two problems have opposite values:
\begin{align}
\minimize{x \in \X_1} \, & f(x) + g(Ax),
\label[problem]{pb:prim_pb_fr} \\
\minimize{\mu \in \X_2^*} \, & f^*(A^* \mu) + g^*(-\mu),
\label[problem]{pb:dual_pb_fr}
\end{align}
where $A^*\colon\X_2^*\to\X_1^*$ is the adjoint operator of $A$. Moreover, if $V<+\infty$, where $V$ denotes the value of \Cref{pb:prim_pb_fr}, then \Cref{pb:dual_pb_fr} has a solution.
\end{theorem}

\Cref{pb:dual_pb_fr} will be called dual problem to \eqref{pb:prim_pb_fr}. When a problem and its dual have opposite values, we say that they are in strong duality. We will call primal-dual gap of \Cref{pb:prim_pb_fr,pb:dual_pb_fr} the quantity $\Delta(x, \mu)$ defined for $(x,\mu)\in\X_1\times\X_2^*$ by
\begin{equation} \label{eq:primal_dual_gap_gal}
\Delta(x,\mu) = \big( f(x) + g(Ax) \big) + \big( f^*(A^* \mu) + g^*(-\mu) \big).
\end{equation}
We next collect some elementary properties of the primal-dual gap.

\begin{corollary}[label=coro:fenchel_rockafellar]
Let $(x,\mu)\in\X_1\times\X_2^*$. Then $\Delta(x,\mu) \geq 0$.
Moreover, if \labelcref{eq:qualif_cond_gal} holds true, then the following three statements are equivalent:
\begin{enumerate}[label = \upshape(\itshape\roman*\,\upshape), ref = \itshape\roman*]
\item\label{ln:null_pdg} $\Delta(x,\mu)=0$,
\item\label{ln:solutions} $x$ is a solution to \Cref{pb:prim_pb_fr} and $\mu$ is a solution to \Cref{pb:dual_pb_fr},
\item\label{ln:opti_cond} $A^*\mu \in \partial f(x)$ and $-\mu \in \partial g(Ax)$,
\item\label{ln:opti_cond_bis} $x \in \partial f^\ast(A^*\mu)$ and $A x \in \partial g^\ast(-\mu)$.
\end{enumerate}
\end{corollary}

\begin{proof}
Observe that, by definition of the adjoint operator,
\begin{equation*}
\Delta(x,\mu) = \big( f(x) + f^*(A^*\mu) - \descrInnerProd{A^*\mu}x{\X_1^*,\X_1} \big)
+ \big( g(Ax) + g^*(-\mu) + \descrInnerProd\mu{Ax}{\X_2^*,\X_2} \big).
\end{equation*}
The nonnegativity of the primal-dual gap then follows from the Fenchel--Young inequality \labelcref{eq:fenchel_young}. 
The equivalence between \labelcref{ln:null_pdg} and \labelcref{ln:solutions} is a consequence of \Cref{th:FenchelRockafellar}. Also, $\Delta(x,\mu)$ is null \emph{iff} both terms in the above decomposition are null. This is equivalent to \labelcref{ln:opti_cond} and to \labelcref{ln:opti_cond_bis}, by \Cref{lem:FenchelYoungEq}, which concludes the proof.
\end{proof}

\paragraph{Bregman distances}

Let $\Xi\in\plscc\E$ be $\beta$-strongly convex, i.e., such that $\Xi-\frac\beta2\Enorm{}{}^2$ is convex.
We denote as $B_\Xi\colon\E\times\dom\Xi\times\E\to\bar\R$ the Bregman distance associated with $\Xi$, which we define for all $\mu\in\E,\mu'\in\dom\Xi$, and $w\in\E$ as
\begin{equation}
\label{eq:bregman_gal}
\xiBreg\mu{\mu'}w=
\begin{dcases*}
+\infty & if $w\not\in\subgrad\Xi{\mu'}$, \\
\Xi(\mu)-\Xi(\mu')-\InnerProd{\mu-\mu'}w & otherwise.
\end{dcases*}
\end{equation}
Let us note that this definition of the Bregman distance is not quite the standard one, in which one usually requires $\Xi$ to be differentiable. In this case, one can eliminate $w$ from the above definition and replace it by $\grad\Xi{\mu'}$, which yields the standard definition.
In particular, when $\Xi=\frac12\sqnorm{}$, we have $B_{\Xi}(\mu,(\mu',w))= \frac12 \sqnorm{\mu'-\mu}+\indic{\{0\}}{w-\mu'}$.

\begin{remark}[label=rem:breg_sqdist]
Notice that, since $\Xi$ is $\beta$-strongly convex, we have for all $\mu\in\E,\mu'\in\dom\Xi$, and $w\in\E$
\begin{equation*}
\xiBreg\mu{\mu'}w\geq\frac\beta2\sqnorm{\mu-\mu'}
\end{equation*}
and thus $\xiBreg\mu{\mu'}w=0$ \textit{iff} $\mu=\mu'$ and $w\in\subgrad\Xi{\mu'}$.
\end{remark}

The following lemma follows from direct calculations.

\begin{lemma}[label=lem:identity_Breg]
For all $a,b\in\dom{\subgrad{\Xi}{}}$, $c\in\dom\Xi$, $w_a\in\subgrad\Xi a$, and $w_b\in\subgrad\Xi b$, we have the identity
\begin{equation*}
\xiBreg cb{w_b}+\xiBreg ba{w_a}-\xiBreg ca{w_a}=\InnerProd{c-b}{w_a-w_b}.
\end{equation*}
\end{lemma}

\paragraph{Coercive functions}

Let $f \colon \X \rightarrow \bar{\R}$ and $\alpha\in\R$. We call sublevel set of $f$ at height $\alpha$ the set
\begin{equation*}
\{f\leq\alpha\}=\left\{x\in\X\,\big|\,f(x)\leq\alpha\right\}.
\end{equation*}
We say that $f$ is coercive if
\begin{equation*}
\converge{f(x)}{+\infty}{\norm x}{+\infty}.
\end{equation*}

\begin{remark}
Notice that a function $f$ is coercive \textit{iff} its sublevel sets are bounded. Also notice that the sublevel sets of a convex function are convex.
\end{remark}

When $\E_1$ and $\E_2$ are two Hilbert spaces, with inner products $\EInnerProd{}{}1$ and $\EInnerProd{}{}2$ respectively, unless specified otherwise, we endow the product space $\E=\E_1\times\E_2$ with the canonical inner product, defined for all $x_1,y_1\in\E_1$ and all $x_2,y_2\in\E_2$ as
\begin{equation*}
\EInnerProd{(x_1,x_2)}{(y_1,y_2)}{}=\EInnerProd{x_1}{y_1}1+\EInnerProd{x_2}{y_2}2.
\end{equation*}
Let $f_1\colon\X_1\to\bar\R^+$ and $f_2\colon\X_2\to\bar\R^+$. We denote their direct sum as $f_1\oplus f_2\colon\X_1\times\X_2\to\bar\R^+$. It is defined, for all $x_1\in\X_1, x_2\in\X_2$ as
\begin{equation*}
f_1\oplus f_2(x_1,x_2) = f_1(x_1)+f_2(x_2).
\end{equation*}
The proof of the following lemma is straightforward.

\begin{lemma}[label=lem:sum_coer]
Let $\E_1$, $\E_2$ be two Hilbert spaces. Let $f_1\colon\E_1\to\bar\R^+$ and $f_2\colon\E_2\to\bar\R^+$ be two coercive functions. Assume that both $f_1$ and $f_2$ are bounded from below. Then $f_1\oplus f_2$ is coercive.
\end{lemma}

We also state here the next lemma on the duality between functions with Lipschitz gradients and strongly convex functions, whose proof can be found in \cite[Theorem~18.15 and Corollary~11.16]{bauschke2011convex}.

\begin{lemma} [label=lem:f_star_coercive]
Let $f\in\plscc\E$ and $\beta>0$. Then
$f$ is Fréchet differentiable and its gradient $\grad f{}$ is $\beta$-Lipschitz continuous if and only if the function $\Fen f{}$ is $1/\beta$-strongly convex.
In this case, $f^*$ is also coercive.
\end{lemma}

\paragraph{Perspective functions}

Let $f\in\plscc\E$. We denote its perspective function as $\tilde f\colon\E\times\R \rightarrow \bar\R^+$ and we define it, following \cite{combettes2018perspective}, as
\begin{equation*}
\tilde f(x,s)
= \,
\begin{dcases*}
sf\left(\frac xs\right) & if $s> 0$, \\
\sup_{y\in\dom f}f(y+x)-f(y) & if $s=0$, \\
+ \infty & otherwise.
\end{dcases*}
\end{equation*}

We shall need in the sequel the following property of perspective functions, which can be found in \cite[Proposition~2.3]{combettes2018perspective}.

\begin{lemma}[label=lem:tildef]
Let $f\in\plscc\E$.
Then $\tilde f\in\plscc{\E\times\R}$. Moreover,
\begin{equation*}
(\tilde f)^*(\mu,z) = \indic{\epi{f^*}}{\mu,-z}.
\end{equation*}
If $f$ is Fréchet-differentiable, then for all $s>0$ and $x \in \E$, $\tilde{f}$ is Fréchet-differentiable at $(x,s)$ and
\begin{equation} \label{eq:tildef}
\grad{\tilde f}{x,s}=\left(\grad f{\frac xs},f\left(\frac xs\right)-\InnerProd{\grad f{\frac xs}}{\frac xs}{}\right).
\end{equation}
Finally, if $\grad f{}$ is continuous at $\frac xs$, then $\nabla{\tilde f}{}$ is continuous at $(x,s)$.
\end{lemma}

\section{The Extended Level-Set method} \label{sec:extension}

We introduce in \Cref{subsec:dualization_fwa} a prototype of the Frank--Wolfe algorithm, which contains as particular cases both the fully-corrective FWA and our DLS method. We give a dual interpretation of this method as a general prototype for a cutting-plane algorithm. In \Cref{subsec:def_algo}, we introduce our extension of the Level-Set method of \cite{lemarechal1995new}, which we call Extended Level-Set (ELS) method. This method is a general cutting-plane-type algorithm, which will later yield the desired extension of the FWA by dualization. We give a convergence result for the ELS method in \Cref{subsec:convergence}.

\subsection{A dual point of view on the fully-corrective FWA} \label{subsec:dualization_fwa}

The FWA aims at solving problems of the form
\begin{problem}{$p$}{pb:FW_primal}
\minimize{x\in\X}f(x)+\indic Kx,
\end{problem}
where $\X$ is a Banach space, $K$ is a nonempty closed convex subset of $\X$, and $f \in \plscc\X$.
The dual problem of \Cref{pb:FW_primal} is
\begin{problem}{$d$}{pb:FW_dual}
\minimize{\mu\in\X^*} f^*(\mu) + \support{-K}{\mu}.
\end{problem}
The primal dual-gap of \Cref{pb:FW_primal,pb:FW_dual} is then given, for $(x,\mu) \in \X \times \X^*$, by
\begin{equation*}
\Delta(x,\mu)= f(x)+\indic Kx + \Fen f\mu + \support{-K}\mu,
\end{equation*}
following \labelcref{eq:primal_dual_gap_gal}. By \Cref{coro:fenchel_rockafellar}, we have $\Delta(x,\mu) \geq 0$.
Let us assume that $0 \in \inter{\dom{f}- K}$. Then \Cref{th:FenchelRockafellar} ensures that \Cref{pb:FW_primal,pb:FW_dual} are in strong duality. Moreover, by \Cref{coro:fenchel_rockafellar}, for any pair $(x,\mu) \in \X \times \X^\ast$, $x$ and $\mu$ are respectively solutions to \Cref{pb:FW_primal,pb:FW_dual} \emph{iff} $\Delta(x,\mu)=0$.

\Cref{alg:fwa} below is a general form of the Frank--Wolfe algorithm.
It relies on the linear minimization oracle LMO$\colon \X^* \to K$, which is such that
\begin{equation}\tag={LMO}
\forall \mu \in \X^*, \quad
\text{LMO}(\mu)\in\argmin{x\in K}\InnerProd\mu x.
\end{equation}
\Cref{alg:fwa} covers the classical fully-corrective FWA, in the case where $f$ is continuously differentiable with a Lipschitz-continuous gradient: it suffices to fix a point $x^0 \in K$, to define $\mu^0= \nabla f(x^0)$, $K^{-1}= \{ x^0 \}$, and finally, at each iteration, to define $\mu^{t+1}$ as $\nabla f(\hat{x}^{t+1})$ in the last step. The \textbf{Dual update} step is facultative and can be omitted in general. Notice that in the case of the fully-corrective FWA, this step is implementable without explicit knowledge of $f^*$ and $\sigma_{-K}$
since, for $\mu^{t}=\nabla f(\hat{x}^{t})$, we have
\begin{equation*}
f^*(\mu^{t}) = \InnerProd{\mu^t}{\hat x^t}
- f(\hat{x}^{t})
\quad \text{and} \quad
\sigma_{-K}(\mu^{t})= -\InnerProd{\mu^t}{v^t},
\end{equation*}
by \Cref{lem:FenchelYoungEq} and \Cref{coro:subgrad_oracle}.

\begin{figure}[htb]
\begin{minipage}{0.45\linewidth}
\begin{algorithm}[H]
\SetKwBlock{Oracle}{Oracle:}{}
\SetKwBlock{DualUpdate}{Dual update:}{}
\SetKwBlock{PrimalUpdate}{Primal update:}{}
\SetKwBlock{DualCandidate}{Dual candidate:}{}
\SetKwBlock{Pruning}{Pruning:}{}
\caption{FWA for \Cref{pb:FW_primal}}
\label[algo]{alg:fwa}
\textbf{Require:} $\mu^0 \in \dom{f^*}$\;
Find $K^{-1}\subset K$ such that \\
\qquad $0 \in \inter{\dom{f}- K^{-1}}$\;
\For{$t=0,\ldots$}{
\emph{Available at iteration $t$:}\\
\quad $\mu^t\in\E$, $K^{t-1}\subset K$\;
\Oracle{
Set $v^t = \text{LMO}(\mu^t)$\;
}
\smallskip
\DualUpdate{
\emph{Optional.}
Take $\hat\mu^t$ a solution to
\begin{alg:equation*} \minimize{\mu\in\IntInt{\mu^0}{\mu^t}}f^*(\mu)+\support{-K}\mu
\end{alg:equation*}
Set $\bar h^t = \Fen f{\hat\mu^t}+\support{-K}{\hat\mu^t}$\;
}
\smallskip
\PrimalUpdate{Set $K^t= K^{t-1} \cup \{ v^t \}$\;
Find a solution $\hat x^t$ of
\begin{alg:problem}{$p^t$}{pb:pt}
\minimize{x\in \cloconv{K^t}}f(x)
\end{alg:problem}
Set $\underline h^t =-f\left(\hat x^{t+1}\right)$\;
}
\smallskip
\DualCandidate{
Generate a new candidate $\mu^{t+1}$.
}
\smallskip
}
\end{algorithm}
\end{minipage}
\hspace{0.03\linewidth}
\begin{minipage}{0.49\linewidth}
\begin{algorithm}[H]
\label[algo]{alg:dfwa}
\caption{Dual FWA for \Cref{pb:FW_dual}}
\SetKwBlock{Oracle}{Oracle:}{}
\SetKwBlock{DualUpdate}{Dual update:}{}
\SetKwBlock{PrimalUpdate}{Primal update:}{}
\SetKwBlock{DualCandidate}{Dual candidate:}{}
\SetKwBlock{Pruning}{Pruning:}{}
\textbf{Require:} $\mu^0 \in \dom{f^*}$\;
Find $K^{-1}\subset K$ such that \\
\qquad $0 \in \inter{\dom{f}- K^{-1}}$\;
\For{$t=0,\ldots$}{
\emph{Available at iteration $t$:}\\
\quad $\mu^t\in\E$, $K^{t-1}\subset K$\;
\Oracle{
Find $v^t\in-\subgrad{\support{-K}{}}{\mu^t}$\;
}
\smallskip
\DualUpdate{
Take a solution $\hat\mu^t$ to
\begin{alg:equation*}
\minimize{\mu\in\IntInt{\mu^0}{\mu^t}}\Fen f\mu+\support{-K}\mu
\end{alg:equation*}
Set $\bar h^t=\Fen f{\hat\mu^t}+\support{-K}{\hat\mu^t}$\;
}
\smallskip
\PrimalUpdate{Set $K^t= K^{t-1} \cup\{v^t\}$\;
Find a solution $\nu^{t+1}$ of
\begin{alg:problem}{$d^t$}{pb:dt}
\minimize{\mu\in\E}\Fen f{\mu}+\support{-K^t}\mu
\end{alg:problem}
Set $\underline h^t = \Fen f{\nu^{t+1}}+\support{-K^t}{\nu^{t+1}}$\;
}
\smallskip
\DualCandidate{
Generate a new candidate $\mu^{t+1}$.
}
\smallskip
}
\end{algorithm}
\end{minipage}
\end{figure}

\Cref{alg:dfwa} is equivalent to \Cref{alg:fwa}. First we note that the two \textbf{Oracle} steps are equivalent, as a direct consequence of \Cref{coro:subgrad_oracle}.
We next notice that \Cref{pb:dt} is the dual of \Cref{pb:pt}. Also, for the same reason as with \Cref{pb:FW_primal,pb:FW_dual}, \Cref{pb:pt,pb:dt} are in strong duality, which ensures that the definitions of $\underline h^t$ in the algorithms are equivalent. We note that, in the case of the fully-corrective FWA, defining $\mu^{t+1}$ as $\nabla f(\hat{x}^{t+1})$ is equivalent to directly define $\mu^{t+1}$ as a solution to \Cref{pb:dt}.

\Cref{alg:dfwa} can be seen as a general cutting-plane method for the dual problem \eqref{pb:FW_dual}. By construction, $-v^t \in \partial \sigma_{-K}(\mu^t)$, so the map $\mu \mapsto \InnerProd\mu{-v^t}$ is a linear lower bound of $\sigma_{-K}$, exact at $\mu= \mu^t$. This implies that $\sigma_{-K^t}$ is a lower approximation of $\sigma_{-K}$, which is exact at the points $\mu^0,\ldots,\mu^t$. If moreover $K^{-1}$ is the convex hull of a finite number of points, then $\sigma_{-K^t}$ is piecewise affine. We will call the map $\mu \mapsto \InnerProd\mu{-v^t}$ a cut, and by extension, we will simply call cut any element of $K^t$.

Let us comment on the role of the quantities $\bar h^t$ and $\underline h^t$. Denote by $V_d$ the value of \Cref{pb:FW_dual}. From the definition of $\bar{h}^t$, we directly see that it is an upper bound of $V_d$. Since $\sigma_{-\tilde{K}^t} \leq \sigma_{-K}$ and since $\underline h^t$ is the value of \Cref{pb:dt}, we deduce that $\underline{h}^t$ is a lower bound of $V_d$. This implies in particular that the candidate $\hat{\mu}^t$ obtained at iteration $t$ of the algorithm is $(\bar{h}^t- \underline{h}_t)$-optimal. We can retrieve this property by noticing that
\begin{equation*}
\bar{h}^t- \underline{h}^t
= \Delta(\hat{x}^t,\hat{\mu}^t).
\end{equation*}
The interpretation of the quantity $(\bar{h}^t- \underline{h}^t)$ as a primal-dual gap is of a key importance for the design of the desired extension of the FWA.

Let us recall our general objective: generalizing the FWA to the case of problems with nonsmooth costs $f$, utilizing the duality with cutting-plane-type algorithms. At a dual level, this means that we do not want to assume $f^*$ to be strongly convex, which does not seem restrictive at the first sight, since for basic cutting-plane methods, $f^*$ is simply the characteristic function of some given closed and bounded feasible set. By basic, we have in mind the methods for which one simply defines the next dual candidate $\mu^{t+1}$ as $\nu^{t+1}$. Though these methods are known to converge, the convergence is in practice slow (see \cite[Section~9.3.2]{bonnans2006numerical}); moreover, to ensure the convergence of the corresponding FWA (obtained by ``back''-dualization) to a minimizer, we need to ensure that the primal-dual gap $\bar{h}^t- \underline h^t$ converges to $0$. In view of our objectives, our attention has focused on the cutting-plane method calld Level-Set method proposed and analyzed in \cite[Section~2.2.1]{lemarechal1995new}, for which the convergence of the primal-dual gap is known. As we already pointed out in the introduction, we need to utilize a double extension of this method, since $f^*$ is restricted to be a characteristic function in \cite[Section~2.2.1]{lemarechal1995new} and since a certain projection step realized for the generation of a novel dual candidate (in the last step of the algorithm) must also be generalized. The next section is dedicated to the generalization of the Level-Set method.

\subsection{Statement of the ELS method}
\label{subsec:def_algo}

The aim of this section is to present an algorithm, which we call Extended Level-Set (ELS) method, to solve problems of the form
\begin{problem}{$D$}{pb:prim}
\minimize{\mu\in\E}\psi(\mu)+\support E\mu,
\end{problem}
where $\E$ is a Hilbert space, $\psi\in\plscc\E$, and $E\subset\E$ is a nonempty closed convex set. As its name suggests, the ELS method is an extension of the Level-Set method proposed in \cite[Section~2.2.1]{lemarechal1995new}, which is itself an extension of the Cutting-Plane Algorithm.

We denote by $V_D$ the value of \Cref{pb:prim}.
The problem dual to \Cref{pb:prim} is
\begin{problem}{$P$}{pb:dual}
\minimize{x\in-E} \Fen\psi x
\end{problem}
and we denote by $V_P$ its value. Given $x$ and $\mu$ in $\E$, according to \labelcref{eq:primal_dual_gap_gal}, the primal-dual gap between \Cref{pb:prim,pb:dual} is
\begin{equation*}
\Delta(x,\mu) = \psi(\mu)+\support E\mu+\Fen\psi x+\indic{-E}x.
\end{equation*}
A direct modification of the proof of \Cref{coro:fenchel_rockafellar} shows that $\Delta(x,\mu)=0$ if and only if $\mu$ is a solution to \Cref{pb:prim}, $x$ is a solution to \Cref{pb:dual}, and $V_P+V_D= 0$.

We fix a $\beta$-strongly convex function $\Xi\in\plscc\E$. Recall that $B_{\Xi}$ denotes the associated Bregman distance, in the sense of the definition \labelcref{eq:bregman_gal}.
The ELS method is described in \Cref{alg:ExtLNN}. Note that the \textbf{Primal update} involves a function called pruning, whose output is a subset of $E$. For the moment, we simply take $E^t=E^{t-1}\cup\left\{v^t\right\}$. We first investigate the convergence of the algorithm in this setting; we will later propose some pruning rules which preserve the convergence speed of our algorithm (see the last paragraph of \Cref{subsec:convergence}).

\begin{algorithm}
\label[algo]{alg:ExtLNN}
\caption{Extended Level-Set method for \Cref{pb:prim}}
\SetKwBlock{Oracle}{Oracle:}{}
\SetKwBlock{DualUpdate}{Dual update:}{}
\SetKwBlock{PrimalUpdate}{Primal update:}{}
\SetKwBlock{DualCandidate}{Dual candidate:}{}
\SetKwBlock{Pruning}{Pruning:}{}
\textbf{Require:} $\mu^0\in\dom{\support E{}}\cap\dom{\subgrad\Xi{}}\cap\dom{\psi}$, $w^0\in\subgrad\Xi{\mu^0}$, $E^{-1} \subset E$, $\lambda\in(0,1)$\;
Set $\bar h^{-1}=+\infty$\;
\For{$t=0,\ldots$}{
\emph{Available at iteration t:} $\mu^t\in\E$, $w^t\in\E$, $E^{t-1}\subset E$, $\bar h^{t-1}\in\bar\R^+$\;
\Oracle{
Find $v^t\in\subgrad{\support E{}}{\mu^t}$\;
}
\smallskip
\DualUpdate{
Set $\bar h^t=\min\left\{\bar h^{t-1},\psi\left(\mu^t\right)+\support {E}{\mu^t}\right\}$\;
}
\smallskip
\PrimalUpdate{
Set
$\underline h^t=\inf_{\mu\in\E}\psi(\mu)+\support {E^{t-1}\cup\left\{v^t\right\}}{\mu}
$\;
Set $E^t =\text{pruning}\left(E^{t-1}\cup\left\{v^t\right\}\right)\subset\cloconv{E^{t-1}\cup\left\{v^t\right\}}$\;
}
\smallskip
\DualCandidate{
Set $\Delta^t = \bar h^t-\underline h^t$\;
Set $\ell^t = \lambda \bar h^t+(1-\lambda)\underline h^t$\;
Set $Q^t = \{\psi+\support{E^t}{}\leq \ell^t\}$\;
Set a new candidate $\mu^{t+1}$ as the solution to:
\begin{alg:problem}{}{pb:projgen}
\minimize{\mu\in Q^t}\xiBreg\mu{\mu^t}{w^t}
\end{alg:problem}
Take $w^{t+1}\in\subgrad\Xi{\mu^{t+1}}$ such that $w^t-w^{t+1} \in N_{Q^t}(\mu^{t+1})$\;
}
\smallskip
}
\end{algorithm}

\begin{remark}
We want to highlight \Cref{alg:ExtLNN} as being a specific instance of \Cref{alg:dfwa}, where $\psi$ plays the role of $\Fen f{}$, $E$ and $E^t$ that of $-K$ and $-K^t$ respectively, and $v^t$ is replaced by its opposite vector.
\end{remark}

We now present the elements which support our claim that this algorithm is an extension of the Level-Set method.
\begin{itemize}
\item In our algorithm, we keep the function $\psi$ as is and take subgradients of $\support E{}$, whereas in \cite[Section~2.2.1]{lemarechal1995new}, subgradients of the whole cost function are taken.
\item The step described in \Cref{pb:projgen} is a projection step of $\mu^t$ onto the set $Q^t$, following the Bregman distance associated with $\Xi$. Note that, when $\Xi$ is differentiable, using a first-order optimality condition for $\mu^{t+1}$ in \eqref{pb:projgen}, we have that $w^{t} - \nabla \Xi(\mu^{t+1}) \in N_{Q^t}(\mu^{t+1})$, which ensures the existence of $w^{t+1}$ in the very last step.
\end{itemize}

We now present a list of hypotheses under which the algorithm is well-defined and converges.
Let us stress that some of these assumptions are not explicit, for example \Cref{assum:lpz-like} below. Though it would be easy to transform these assumptions into explicit ones by slightly weakening them, we recall that our main interest does not lie in the ELS method as such but rather in its dual counterpart, our DLS method, for which explicit assumptions will be made later on.

\begin{enumerate}[label = \upshape (H\arabic*), ref = \upshape H\arabic*]
\item\label[assum]{assum:mu_0} We have $\dom{\support E{}}\cap\dom{\subgrad\Xi{}}\cap\dom\psi\neq\varnothing$.
\item\label[assum]{assum:subgrad} For all $t\in\N$, $\subgrad{\support E{}}{\mu^t}\neq\varnothing$. There exists a constant $C_{\oracle}$ such that, for all $t\in\N$, the vector $v^t$ verifies $\norm{v^t}\leq C_{\oracle}$.
\item\label[assum]{assum:psi_growth} The set $E^{-1}$ required at the beginning of the algorithm is bounded and such that $\psi+\support {E^{-1}}{}$ is coercive.
\end{enumerate}

We fix a constant $C_{E^{-1}}$ such that $E^{-1}\subset B(0,C_{E^{-1}})$.
We denote as $Q^{-1}$ the set $\left\{\psi+\support{E^{-1}}{}\leq\bar h^0\right\}.$ This set is convex, nonempty as a consequence of \Cref{assum:mu_0}, and bounded as a consequence of \Cref{assum:psi_growth}. Let then $C_{Q^{-1}}>0$ be such that $Q^{-1}\subset B(0,C_{Q^{-1}})$.

\begin{enumerate}[resume*]
\item\label[assum]{assum:xi} We have $Q^{-1}\subset\dom\Xi$ and
the function $\Xi$ is bounded over the set $Q^{-1}$ by a constant $C_{\Xi,Q^{-1}}>0$.
\item\label[assum]{assum:w} For all $t\in\N$, there exists $w^{t+1}\in\subgrad\Xi{\mu^{t+1}}$ such that
$w^t-w^{t+1} \in N_{Q^t}(\mu^{t+1})$. There exists a constant $C_{\subgrad\Xi{}}>0$ such that for all $t\in\N$, $\norm{w^t}\leq C_{\partial\Xi}$.
\item\label[assum]{assum:lpz-like} There exists a constant $C_\psi > 0$ such that, for all $t\in\N$,
\begin{equation*}
\abs*{\psi\left(\mu^{t+1}\right)-\psi\left(\mu^t\right)}\leq C_\psi\norm{\mu^{t+1}-\mu^t}.
\end{equation*}
\end{enumerate}

In the rest of this section, we assume that all these assumptions are verified and we recall that, except for the last paragraph of this section, we assume that we take $E^t = E^{t-1} \cup \{v^t\}$ in the \textbf{Primal update} step of \Cref{alg:ExtLNN}. We next provide some first properties of \Cref{alg:ExtLNN}.

\begin{prop}[label=prop:well_posed_els]
The Extended Level-Set method from \Cref{alg:ExtLNN} is well-posed and, for any $t \in \mathbb{N}$, we have
\begin{equation} \label{eq:incr_decr}
-\infty<\underline h^t\leq\underline h^{t+1}\leq -V_P\leq V_D\leq \bar h^{t+1}\leq\bar h^t<+\infty.
\end{equation}
Moreover, for any $t \in \mathbb{N}$, there exist two elements $\hat{x}^t$ and $\hat{\mu}^t$ such that
\begin{equation}\label{eq:candidates}
\hat x^t\in \argmin{x\in\E}\Fen\psi x+\indic{-\cloconv{E^{t-1}\cup\left\{v^t\right\}}}x\quad \text{and}\quad \hat\mu^t\in\argmin{\mu\in\IntInt{\mu^0}{\mu^t}}\psi(\mu)+\support E\mu.
\end{equation}
We also have
\begin{equation}\label{eq:errors}
-\underline h^t=\Fen\psi{\hat x^t} + \iota_{-E}\left( \hat x^t \right),
\quad
\bar h^t=\psi\left(\hat\mu^t\right)+\support E{\hat\mu^t},
\quad \text{and} \quad
\Delta^t= \Delta\left(\hat x^t,\hat\mu^t\right).
\end{equation}
Finally, the sequence $(\Delta^t)_{t \in \mathbb{N}}$ is nonincreasing.
\end{prop}

\begin{proof}
We do a proof by induction. We claim that, for any $t \in \mathbb{N}$, the algorithm can be run until the beginning of iteration $t$ and that the following is satisfied:
\begin{equation*}
\mu^t \in \dom{\psi},
\quad
w^t \in \partial \Xi(\mu^t),
\quad
E^{-1} \subset \cloconv{E^{t-1}},
\quad
\text{and} \quad
t\geq 1 \Rightarrow \bar{h}^{t-1} \leq \bar{h}^0.
\end{equation*}
We also claim that $E^{t-1}$ is a bounded subset of $E$.
For $t=0$, the claim follows from \Cref{assum:mu_0}. Let us assume that it is satisfied for some $t \in \mathbb{N}$ and let us consider the execution of the iteration $t$ of the method.

The \textbf{Oracle} step is well-defined, by \Cref{assum:subgrad}, which also implies that $\mu^t \in \dom{\sigma_E}$.
Concerning the \textbf{Dual update} step, since $\mu^t \in \dom{\sigma_E} \cap \dom{\psi}$, we have $\psi(\mu^t) + \sigma_E(\mu^t)< + \infty$ and thus $\bar{h}^t < \infty$.
If $t = 0$, then $\bar{h}^t= \bar{h}^0$. If $t \geq 1$, then $\bar{h}^t \leq \bar{h}^{t-1} \leq \bar{h}^0$.

Let us move to the \textbf{Primal update} step. Since $E^{-1} \subset \cloconv{E^{t-1} \cup \{ v^t \}}$, we have that $\psi + \sigma_{E^{-1}} \leq \psi + \sigma_{E^{t-1} \cup \{ v^t \}}$, which implies that $\psi + \sigma_{E^{t-1} \cup \{ v^t \}}$ is coercive, by \Cref{assum:psi_growth}. Since this function is also convex and lower semicontinuous, it has a minimizer $\nu^t$ and $\underline{h}^t$ is finite, by \cite[Proposition~11.14]{bauschke2011convex}. Since $E^{t-1} \cup \{ v^t \}$ is bounded, then so is $E^t$, and we have $\dom{\sigma_{E^{t-1} \cup \{ v^t \}}}= \mathcal{E}$. Applying \Cref{th:FenchelRockafellar}, we deduce that
\begin{equation*}
-\underline h^t = \inf_{x\in\E}\Fen\psi x+\indic{-\cloconv{E^{t-1}\cup\left\{v^t\right\}}}x
\end{equation*}
and that the above problem has a solution $\hat{x}^t$. This implies that $\underline h^t \leq - V_P \leq V_D$. By construction, $\hat{x}^t \in - \cloconv{E^{t-1}\cup\left\{v^t\right\}} \subset - E$, which implies that $-\underline{h}^t=\Fen\psi{\hat x^t}+\indic{-E}{\hat x^t}$. From the definition of $E^t$, we have that $E^{-1} \subset \cloconv{E^t}$.

Finally, we discuss the \textbf{Dual candidate} step. We have $\Delta^t \geq 0$, because $\bar{h}^t \geq V_D \geq -V_P$ and $\underline{h}^t \leq- V^P$.
By definition of $\ell^t$, we have $\underline h^t\leq\ell^t \leq \bar h^t$. Since moreover $E^t \subset \cloconv{E^{t-1} \cup \{ v^t \}}$, we have
\begin{equation*}
\psi(\nu^t) + \sigma_{E^t}(\nu^t) \leq \psi(\nu^t) + \sigma_{E^{t-1} \cup \{ v^t \}}(\nu^t) = \underline{h}^t \leq \ell^t,
\end{equation*}
which implies that $\nu^t \in Q^t$ and thus $Q^t$ is nonempty. Since $\bar{h}^t \leq \bar{h}^0$ and $\sigma_{E^t} \geq \sigma_{E^{-1}}$, we deduce that $Q^t \subset Q^{-1} \subset \dom{\Xi}=\dom{B_{\Xi}(\cdot,(\mu^t,w^t))}$, by \Cref{assum:xi}. This implies that \Cref{pb:projgen} has a solution $\mu^{t+1}$. Note that $\mu^{t+1} \in Q^t \subset \dom{\psi}$. Finally, the existence of $w^{t+1} \in \partial \Xi(\mu^{t+1}) \cap \big( w^t - N_{Q^t}(\mu^{t+1}) \big)$ is ensured by \Cref{assum:w}. Therefore, the claim is verified for $t+1$, which proves the well-posedness of the algorithm.

Concerning the proof of \labelcref{eq:incr_decr}, it remains to prove that $\bar{h}^t$ is nonincreasing and that $\underline{h}^t$ is nondecreasing. It is obvious that $\bar{h}^t$ is nonincreasing. The fact that $\underline{h}^t \leq \underline{h}^{t+1}$ is a consequence of the inclusion $E^{t-1} \cup \{v^t\} \subset \cloconv{E^{t} \cup \{ v^{t+1} \}}$, which ensures that $\sigma_{E^{t-1} \cup \{ v^t \}} \leq \sigma_{E^{t} \cup \{ v^{t+1} \} }$.

Next, concerning \labelcref{eq:candidates,eq:errors}, we have already proved the existence of $\hat{x}^t$ and we have already justified that $-\underline h^t=\Fen\psi{\hat x^t} + \iota_{-E}\left( \hat x^t \right)$. The existence of $\hat{\mu}^t$ and the fact that $\bar h^t=\psi\left(\hat\mu^t\right)+\support E{\hat\mu^t}$ is straightforward. The equality $\Delta^t= \Delta\left(\hat x^t,\hat\mu^t\right)
$ is then a simple consequence of the definition of the primal-dual gap. Finally, $(\Delta^t)_{t \in \mathbb{N}}$ in nonincreasing because $(\bar{h}^t)_{t \in \mathbb{N}}$ and $(\underline{h}^t)_{t \in \mathbb{N}}$ are respectively nonincreasing and nondecreasing. This concludes the proof of the proposition.
\end{proof}

\begin{remark}
\Cref{pb:projgen} has a unique solution since $\Xi$ is strongly convex and $Q^t$ is convex.
\end{remark}

\subsection{Convergence analysis}\label{subsec:convergence}

The aim of this subsection is to prove the convergence of \Cref{alg:ExtLNN} under the previous assumptions and provide its convergence speed. We start by noticing that, exactly as in the proof of the lemma inside \cite[Theorem~2.2.1]{lemarechal1995new}, we have the following property.

\begin{lemma}[label=lem:LNN_step3]
Let $t_1\leq t_2$ be such that
$\Delta^{t_2}\geq (1-\lambda)\Delta^{t_1}$.
Then $\underline h^{t_2}\leq \ell^{t_1}$.
\end{lemma}

The next theorem deals with the convergence of \Cref{alg:ExtLNN}. Its proof follows closely the analysis in \cite[Theorem~2.2.1]{lemarechal1995new}, but additional care is needed in order to take into account our generalizations, in particular the use of the Bregman distance $B_{\Xi}$ for the projection step.
We hence provide a detalied proof below, which makes use of the constants $C_{\oracle}$, $C_{E^{-1}}$, $C_{Q^{-1}}$, $C_{\Xi,Q^{-1}}$, $C_{\partial\Xi}$, and $C_\psi$, introduced with \Crefrange{assum:subgrad}{assum:lpz-like}.

\begin{theorem}[label=th:cv_LM]
Consider the Extended Level-Set method from \Cref{alg:ExtLNN} under \Crefrange{assum:mu_0}{assum:lpz-like} and with $E^t = E^{t-1} \cup \{v^t\}$ in the \textbf{\upshape{Primal update}} step.
The primal-dual gap $\Delta^t$ converges to $0$ with a speed of order $1/\sqrt t$, i.e., there exists $C>0$ such that, for all $t\in\N^*$,
\begin{equation*}
\Delta^t\leq\frac C{\sqrt t}.
\end{equation*}
\end{theorem}

\begin{proof}
\step{1} Let $T\in\N$ and set $\varepsilon=\Delta^T$ and $I=\IntInt0T$. We recall that, using the monotonicity of $\left(\Delta^t\right)_{t\in\N}$, we have $\varepsilon=\displaystyle\inf_{t\in I} \Delta^t$. We split $I$ in a partition $I_1,\dotsc,I_m$ as follows:
\begin{itemize}
\item We set $p=0$ and $i_0=-1$.
\item While $i_p<T$, we set
\begin{equation*}
i_{p+1}=\max\left\{t\in\IntInt0T\,\big|\,\Delta^t\geq(1-\lambda)\Delta^{i_p+1}\right\}\quad \text{and}\quad I_{p+1}=\IntInt{i_p+1}{i_{p+1}}
\end{equation*}
and we increment $p$ by $1$.
\end{itemize}
Following \cite{lemarechal1995new}, for all $p\in\IntInt0{m-1}$, the iteration $i_p+1$ is called \textit{critical}.
Notice that, using the monotonicity of the sequence $\left(\Delta^t\right)_{t\in\N}$, we have, for all $p\in\IntInt1m$ and $t\in I_p$,
\begin{equation}
\Delta^t\geq(1-\lambda)\Delta^{i_{p-1}+1}.
\end{equation}

Now, let $p\in\IntInt1m$ and $\chi^p$ be a minimizer of $\psi +\support{E^{i_p}}{}$. Notice that such a minimizer exists, since $\psi+\support{E^{i_p}}{}\in\plscc{\E}$ and is coercive using \Cref{assum:psi_growth}.
Then, \Cref{lem:LNN_step3} applied with $t_1=t\in I_p$ and $t_2=i_p$ shows that 
\begin{equation*}
\psi\left(\chi^p\right)+\support{E^{i_p}}{\chi^p}=\underline h^{i_p}\leq\ell^t.
\end{equation*}
Since for all $t\in I_p$, we have $\support{E^t}{}\leq\support{E^{i_p}}{}$, this yields
$\chi^p\in\bigcap_{t\in I_p}Q^t$.
This construction holds for any $p\in\IntInt1m$.

\step{2} Let $p\in\IntInt1m.$ For all $t\in\N$, we set $\tau_p^t=\xiBreg{\chi^p}{\mu^t}{w^t}\geq0.$
Now, let $t\in I_p$.
We use the identity from \Cref{lem:identity_Breg}
with $a=\mu^t$, $b=\mu^{t+1}$, $c=\chi^p$, $w_a=w^t$, and $w_b=w^{t+1}$, which yields
\begin{equation*}
\tau_p^{t+1}+\xiBreg{\mu^{t+1}}{\mu^t}{w^t}-\tau_p^t=\InnerProd{w^t-w^{t+1}}{\chi^p-\mu^{t+1}}.
\end{equation*}
Reordering and using \Cref{assum:w} yields
\begin{equation}\label{eq:reltau1}
0\leq\tau_p^{t+1}\leq\tau_p^t-\xiBreg{\mu^{t+1}}{\mu^t}{w^t}.
\end{equation}
In turn, using \Cref{rem:breg_sqdist}, we have
\begin{equation}\label{eq:reltau2}
0\leq\tau_p^{t+1} \leq \tau_p^t-\frac\beta2\sqnorm{\mu^t-\mu^{t+1}} \leq \tau_p^t-\frac\beta{2C^2}\sqabs*{\psi\left(\mu^{t+1}\right)+\support{E^t}{\mu^{t+1}}-\psi\left(\mu^t\right)-\support{E^t}{\mu^t}},
\end{equation}
where $C=C_\psi+\max(C_{E^{-1}},C_{\oracle})$. Indeed,
\begin{align*}
\abs*{\psi\left(\mu^{t+1}\right)+\support{E^t}{\mu^{t+1}}-\psi\left(\mu^t\right)-\support{E^t}{\mu^t}}&\leq\abs*{\psi\left(\mu^{t+1}\right)-\psi\left(\mu^t\right)}+\abs*{\support{E^t}{\mu^{t+1}}-\support{E^t}{\mu^t}}\\
&\leq C\norm{\mu^{t+1}-\mu^t},
\end{align*}
using \Cref{assum:lpz-like} and the fact that $E^t\subset B(0,\max(C_{E^{-1}},C_{\oracle}))$, which is itself a consequence of the definition of $E^t$ and of \Cref{assum:subgrad,assum:psi_growth}.
Moreover, we know that
\begin{equation}\label{eq:reltau3}
\abs*{\psi\left(\mu^{t+1}\right)+\support{E^t}{\mu^{t+1}}-\psi\left(\mu^t\right)-\support{E^t}{\mu^t}}\geq (1-\lambda)\Delta^t,
\end{equation}
since
\begin{align*}
\psi\left(\mu^t\right)+\support{E^t}{\mu^t}-\psi\left(\mu^{t+1}\right)-\support{E^t}{\mu^{t+1}} & \geq\psi\left(\mu^t\right)+\support{E^t}{\mu^t}-\ell^t\\
&\geq \bar h^t-\ell^t\\
& =(1-\lambda)\Delta^t,
\end{align*}
where we use, in order, the fact that $\psi\left(\mu^{t+1}\right)+\support{E^t}{\mu^{t+1}}\leq\ell^t$ and the definitions of $\bar h^t$ and of $\Delta^t$.
Combining \cref{eq:reltau2,eq:reltau3} yields
\begin{equation*}
0\leq\tau_p^{t+1}\leq\tau_p^t-\frac\beta{2C^2}\left((1-\lambda)\Delta^t\right)^2,
\end{equation*}
which implies, using the nonnegativeness and nonincreasingness of the sequence $\left(\Delta^t\right)_{t\in\N}$,
\begin{equation}\label{eq:rel_tau}
0\leq\tau_p^{t+1}\leq \tau_p^t-\frac\beta{2C^2}\left((1-\lambda)\Delta^{i_p}\right)^2.
\end{equation}
Taking $t=i_p$ and iterating $i_p-i_{p-1}-1$ times \cref{eq:rel_tau} yields
\begin{equation}\label{eq:bound_lastfirst}
0\leq\tau_p^{i_p+1}\leq \tau_p^{i_{p-1}+1}-(i_p-i_{p-1})\frac\beta{2C^2}\left((1-\lambda)\Delta^{i_p}\right)^2.
\end{equation}
Notice also that 
\begin{equation}\label{eq:bound_tau}
\tau_p^{i_{p-1}+1}\leq 2\left(C_{\Xi,Q^{-1}}+C_{\subgrad\Xi{}}C_{Q^{-1}}\right),
\end{equation}
since, for all $t\in\N$,
\begin{equation*}
\tau_p^t\leq\abs*{\Xi\left(\chi^p\right)-\Xi\left(\mu^t\right)}+\abs*{\InnerProd{w^t}{\chi^p-\mu^t}}.
\end{equation*}
Combining \cref{eq:bound_lastfirst,eq:bound_tau} yields
\begin{equation*}
0\leq2\left(C_{\Xi,Q^{-1}}+C_{\subgrad\Xi{}}C_{Q^{-1}}\right)-(i_p-i_{p-1})\frac\beta{2C^2}\left((1-\lambda)\Delta^{i_p}\right)^2,
\end{equation*}
and thus
\begin{equation}\label{eq:bound_sizeIp}
\abs*{I_p}=i_p-i_{p-1}\leq\frac{4 C^2\left(C_{\Xi,Q^{-1}}+C_{\subgrad{\Xi}{}}C_{Q^{-1}}\right)}{\beta\left((1-\lambda)\Delta^{i_p}\right)^2}=\frac {\bar C}{\left((1-\lambda)\Delta^{i_p}\right)^2}.
\end{equation}

\step{3} Since we have, by definition of the indices $i_p$ and nonincreasingness of $\left(\Delta^t\right)_{t \in \N}$,
\begin{equation*}
\Delta^{i_{m-1}+1}\geq\Delta^{i_m}=\Delta^T=\varepsilon\quad\text{and}\quad \Delta^{i_{p+1}}\geq (1-\lambda)\Delta^{i_p+1}>\Delta^{i_{p+1}+1} \text{ for all }p\in\IntInt1{m-2},
\end{equation*}
then for all $p\in\IntInt1{m-1}$, we have
\begin{equation}
\Delta^{i_p}\geq\frac{\varepsilon}{(1-\lambda)^{m-1-p}}.
\end{equation}
Summing \cref{eq:bound_sizeIp} then yields
\begin{align}
\notag T+1&=\sum_{p=1}^m\abs*{I_p}
\leq \frac{\bar C}{(1-\lambda)^2}\sum_{p=1}^m\left(\frac1{\Delta^{i_p}}\right)^2
\leq \frac{\bar C}{(1-\lambda)^2}\left(\frac1{\varepsilon^2}+\sum_{p=1}^{m-1}\frac{(1-\lambda)^{2(m-1-p)}}{\varepsilon^2}\right)\\
&\leq\frac{\bar C}{(1-\lambda)^2\varepsilon^2}\left(1+\sum_{p\in\N}(1-\lambda)^{2p}\right)
\label{eq:lambda}=\frac{\bar C}{(1-\lambda)^2\varepsilon^2}\left(1+\frac1{\lambda(2-\lambda)}\right),
\end{align}
which shows the expected result.
\end{proof}

\begin{remark}
We can minimize \cref{eq:lambda} with respect to $\lambda$, which gives $\bar\lambda=1-\sqrt{2-\sqrt2}\approx0.23$.
\end{remark}

\paragraph{The pruning step}

As is, the main issue with the Extended Level-Set method from \Cref{alg:ExtLNN} is that we need to keep track of all the subgradients $v^t$. This may become costly in terms of memory and of computation time of $\underline h^t$ and of $\mu^{t+1}$, as the set $E^t$ appears in the definition of $Q^t$. This is why we propose to apply pruning steps in our algorithm. By pruning, we mean that we want to define $E^t$ as a smaller set than $E^{t-1} \cup \{ v^t \}$, in the sense that $E^t \subset \cloconv{E^{t-1} \cup \{ v^t \}}$. Note that imposing $E^t \subset \cloconv{E^{t-1} \cup \{ v^t \}}$ at each iteration implies that $E^t \subset \cloconv{E^{-1} \cup\IntInt{v^0}{v^t}}$ and gives thus the boundedness of $E^t$.

Pruning offers the possibility to chose a set $E^t$ with a small cardinality, so as to simplify the implementation of the \textbf{Primal update} and \textbf{Dual candidate} steps. We establish in this paragraph some sufficient properties on the choice of $E^t$ which ensure that \Cref{alg:ExtLNN} remains well-posed and that its convergence properties are preserved. The proof of \Cref{prop:well_posed_els} reveals that the algorithm indeed remains well-posed if we require that $E^{-1}\subset \cloconv{E^t}$.
The convergence proof of \Cref{th:cv_LM} still holds if the following holds.
\begin{enumerate}[label = \upshape(\itshape\roman*\,\upshape), ref = \itshape\roman*]
\item\label{prun:monot} The sequences $\left(\bar h^t\right)_{t\in\N}$ and $\left(\underline h^t\right)_{t\in\N}$ keep the same monotonicity as shown in \labelcref{eq:incr_decr}.
\item\label{prun:coer} The function $\psi+\support{E^t}{}$ remains coercive.
\item\label{prun:common} We are able to find $\chi^p\in\displaystyle\bigcap_{t\in I_p}Q^t$, where $I_p$ describes a subinterval defined in the proof of \Cref{th:cv_LM}.
\end{enumerate}

For property (\labelcref{prun:common}\,) to hold, we decide to only make pruning steps at the critical iterations. Notice that detecting a critical iteration is easy, since it only requires to keep track of $\Delta^j$, where $j$ denotes the last critical iteration. For property (\labelcref{prun:coer}\,) to hold, we only need to keep $E^{-1}\subset \cloconv{E^t}$ after the pruning step, as we already required for the well-posedness of the algorithm. Lastly, for property (\labelcref{prun:monot}\,), notice that the monotonicity of the sequence $\left(\bar h^t\right)_{t\in\N}$ is preserved, and for the monotonicity of $\left(\underline h^t\right)_{t\in\N}$, it suffices that $\underline h^t$ be preserved by pruning steps, i.e., that, for all $t\in\N$,
\begin{equation}\label{eq:monot}
\underline h^t = \tilde{h}^t := \inf_{\mu\in\E}\psi(\mu)+\support{E^t}\mu.
\end{equation}
By definition, this equality holds at noncritical iterations. Let then $t$ be a critical iteration. Since $E^t\subset \cloconv{E^{t-1}\cup\{v^t\}}$, we have $\bar{h}^t \geq \tilde{h}^t$.
Since $E^t$ is bounded, we have
\begin{equation}
\tilde{h}^t = -\min_{x\in\E}\Fen\psi x+\indic{-\cloconv{E^t}}{x}.
\end{equation}
We recall that we proved in \Cref{prop:well_posed_els} that $\underline{h}_t= - \psi^*(\hat{x}^t)$ and that $\hat{x}^t \in \cloconv{E^{t-1} \cup \{ v^t \}}$. Therefore, to ensure that $\tilde{h}^t \geq \bar{h}^t$ (and thus for property (\labelcref{prun:monot}\,) to hold) it suffices to require that $\hat{x}^t \in \cloconv{E^t}$.

We summarize the previous discussion by presenting, in \Cref{alg:ExtLNN_prun}, the extension of \Cref{alg:ExtLNN} with a pruning step satisfying the above requirements, and we deduce at once the following convergence result.

\begin{algorithm}
\SetKwBlock{Oracle}{Oracle:}{}
\SetKwBlock{DualUpdate}{Dual update:}{}
\SetKwBlock{PrimalUpdate}{Primal update:}{}
\SetKwBlock{DualCandidate}{Dual candidate:}{}
\SetKwBlock{Pruning}{Pruning:}{}
\textbf{Require:} $\mu^0\in\dom{\support E{}}\cap\dom{\subgrad\Xi{}}\cap\dom{\psi}$, $w^0\in\subgrad\Xi{\mu^0}$, $E^{-1} \subset E$, $\lambda\in(0,1)$\;
Set $\bar h^{-1}=+\infty$ and $\bar\Delta = +\infty$\;
\For{$t=0,\ldots$}{
\emph{Available at iteration $t$:} $\mu^t$, $w^t$, $E^{t-1}$, $\bar h^{t-1}$ as in \Cref{alg:ExtLNN}, and $\bar\Delta\in\R_+\cup\{+\infty\}$\;
\Oracle{
Find $v^t\in\subgrad{\support E{}}{\mu^t}$\;
Set $\tilde{E^t}=\cloconv{E^{t-1}\cup{v^t}}$\;
}
\smallskip
\DualUpdate{
Set $\bar h^t$, $\underline h^t$, $\Delta^t$, $\ell^t$ as in \Cref{alg:ExtLNN}, and take $\hat x^t$ as a solution to
\begin{alg:equation*}
\minimize{x\in\E}\Fen\psi x+\indic{-\tilde{E^t}}x\;
\end{alg:equation*}
}
\smallskip
\Pruning{
\eIf{$\Delta^t<(1-\lambda)\bar\Delta$}{Take
$E^t\subset \cloconv{E^{t-1}\cup\left\{v^t\right\}}$ such that
$\left\{-\hat x^t\right\}\cup E^{-1}\subset \cloconv{E^t}$\;
Set $\bar\Delta = \Delta^t$\;
}
{
Set $E^t = E^{t-1}\cup\left\{v^t\right\}$\;
}
}
\smallskip
\DualCandidate
{
Take $Q^t$, $\mu^{t+1}$, and $w^{t+1}$ as in \Cref{alg:ExtLNN}\;
}
\smallskip
}
\caption{Extended Level-Set method for \Cref{pb:prim} with pruning}
\label[algo]{alg:ExtLNN_prun}
\end{algorithm}

\begin{theorem}[label=th:cv_LMP]
Consider the Extended Level-Set method with pruning from \Cref{alg:ExtLNN_prun} under \Crefrange{assum:mu_0}{assum:lpz-like}. The primal-dual gap $\Delta^t$ converges to $0$ with a speed of order $1/\sqrt t$.
\end{theorem}

\section{The Dualized Level-Set method}\label{sec:DLS}

We present in this section the general nonsmooth problem that we aim at solving with our DLS algorithm, \Cref{pb:dep_DLS}. Our approach consists in applying the ELS method from \Cref{alg:ExtLNN_prun} to its dual, \Cref{pb:now_prim}. 

\subsection{Framework and mathematical assumptions}
\label{subsec:framework_DLS}

Let $\E_1$ and $\E_2$ be Hilbert spaces and $\E$ their product space. We aim at solving problems of the form
\begin{problem}{\textbf{P}}{pb:dep_DLS}
\minimize{x_1\in\E_1}f(x_1)+\support Q{Ax_1-b}+\indic K{x_1},
\end{problem}
where $A \colon \E_1 \to \E_2$ is a linear operator, $b \in \E_2$, $Q \subset \E_2$, and $K \subset \E_1$. We denote by $V$ the value of \Cref{pb:dep_DLS}.

\paragraph{Structural assumptions}

We make the following assumptions.
\begin{enumerate}[label = \upshape(A\arabic*), ref = \upshape A\arabic*]
\item\label[assum]{struct:lpz} The function $f\colon\E_1\to\R$ is convex and has $\beta$-Lipschitz gradient, for some $\beta>0$.
\item\label[assum]{struct:shape_Q} The set $Q\subset \E_2$ is nonempty and can be decomposed in $Q=Q_1+Q_2$, where $Q_1$ is a closed convex bounded set and $Q_2$ is a closed convex cone. Let then $C_{Q_1}$ be such that $Q_1\subset B(0,C_{Q_1})$.
\item\label[assum]{struct:shape_K} The set $K \subset \E_1$ is nonempty, closed, and convex.
\item\label[assum]{struct:shape_A} The operator $A\colon\E_1\to\E_2$ is linear and bounded.
\item\label[assum]{struct:mu_0} There exists $\mu^0 \in \dom{\subgrad{\Fen{f}{}}{}} \times Q$ such that $\mu^0_1 + A^\ast \mu^0_2 \in \dom{\support{-K}{}}$.
\end{enumerate}
Notice that, in this context, we have
\begin{equation*}
\dom{\support Q{\cdot - b}}=b+Q_2^\ominus,
\quad
\text{where }
Q_2^\ominus=\left\{x_2\in\E_2\,\big|\,\forall \mu_2\in Q_2,\EInnerProd{\mu_2}{x_2}2\leq 0\right\}.
\end{equation*}
Moreover, \Cref{pb:dep_DLS} is equivalent to
\begin{problem}{}{pb:int_DLS}
\minimize{x\in\E}f(x_1)+\support Q{x_2-b}+\indic{\mathbf K}x,
\end{problem}where 
\begin{equation*}
\mathbf K = \left\{x\in\E\,\big|\,x_1\in K,\, x_2=Ax_1\right\},
\end{equation*}
in the sense that the value of \Cref{pb:int_DLS} is $V$ and $x$ is solution to \Cref{pb:int_DLS} \textit{iff} $x_2=Ax_1$ and $x_1$ is solution to \Cref{pb:dep_DLS}.

\paragraph{Qualification condition}

We require the following qualification condition.
\begin{enumerate}[resume*]
\item\label[assum]{assum:qualif} There exists a compact set $K^{-1}\subset K$ such that
\begin{equation}\label{eq:qualif}
-b\in\inter{Q_2^\ominus-A\,\cloconv{K^{-1}}}.
\end{equation}
\end{enumerate}

\begin{remark}
Note that, if $Q$ is bounded, then $Q_2 = \{0\}$, so that $Q_2^\ominus = \E_2$ and the qualification condition \eqref{eq:qualif} is satisfied.
\end{remark}

The dual problem to \Cref{pb:int_DLS} writes:
\begin{problem}{\textbf{D}}{pb:now_prim}
\minimize{\mu\in\E}\Fen f{\mu_1}+\indic Q{\mu_2}+\EInnerProd{\mu_2}b2+\support E\mu, \quad
\text{where } E=-\mathbf K. 
\end{problem}

\begin{lemma}
Under \Crefrange{struct:lpz}{assum:qualif}, the value of \Cref{pb:now_prim} is equal to $-V$. Moreover, if $V<+\infty$, then \Cref{pb:now_prim} has a solution.
\end{lemma}

\begin{proof}
By \Cref{th:FenchelRockafellar}, it suffices to verify that
$0_\E \in \inter{\dom{f} \times \dom{\support Q{\cdot-b}} -\mathbf K}$. The inclusion \eqref{eq:qualif} indeed implies that 
\begin{align*}
0_\E \in\E_1\times\inter{b+Q_2^\ominus-A\,\cloconv{K^{-1}}}&\subset\E_1\times\inter{b+Q_2^\ominus-AK}\\
&=\inter{\E_1\times \left(b+Q_2^\ominus-AK\right)}\\
&=\inter{\left(\E_1\times\left(b+Q_2^\ominus\right)\right)-\mathbf K},
\end{align*}
as was to be proved.
\end{proof}

We define $\psi\colon\E\to\bar\R$ as, for all $\mu\in\E$
\begin{equation}
\label{eq:psi-we-use}
\psi(\mu)=\Fen f{\mu_1}+\indic Q{\mu_2}+\EInnerProd{\mu_2}b2
\end{equation}
so that \Cref{pb:now_prim} has the form of \Cref{pb:prim}.
For future reference, we let $C_{K^{-1}}>0$ be such that $K^{-1}\subset B(0,C_{K^{-1}})$ and we set
\begin{equation*}
E^{-1}=\left\{x\in\E\,\big|\,x_1\in -K^{-1},x_2=Ax_1\right\}.
\end{equation*}

\subsection{Numerical assumptions and statement of the algorithm}
\label{sec:statement-algorithm}

Before providing the statement of the DLS algorithm, we list the required numerical assumptions for its implementation.

\begin{enumerate}[label = (\upshape N\arabic*), ref = \upshape N\arabic*]
\item\label[assum]{assum:oracle} For a given $\mu_1\in\E_1$, we are able to find efficiently a solution to the problem
\begin{problem}{}{pb:oracle_dls}
\minimize{v_1\in K}\EInnerProd{\mu_1}{v_1}1.
\end{problem}
In other words, we have a LMO on $K$ at $\mu_1$.

\item For a given $\mu_2\in \E_2$, we can find efficiently a vector $\bar\mu_2\in\argmin{\nu_2\in Q}\Enorm{\mu_2-\nu_2}2$. In other words, we can project efficiently onto $Q$. Notice that this also gives an easy access to $\dist{\mu_2}Q$.

\item\label[assum]{assum:simple_sets} For any simple set $K'$, we can solve efficiently the problem
\begin{problem}{}{pb:optim_combi_conv}
\minimize{x_1\in\E_1}f(x_1)+\support Q{Ax_1-b}+\indic{\cloconv{K'}}{x_1},
\end{problem}
where we say that $K'$ is simple if it is of the form $K^{-1}\cup S$ for some finite set $S$.
\end{enumerate}

Before we state our last numerical assumption, we specify the Bregman distance $B_{\Xi}$ that will be used and state a few basic properties that it satisfies. We define $\Xi\colon\E\to\bar\R$ as, for all $\mu\in\E$,
\begin{equation}
\label{eq:def-Xi}
\Xi(\mu)=\Fen f{\mu_1}+\frac12\sqEnorm{\mu_2}2.
\end{equation}

\begin{remark}[label=rem:subgrad_and_breg]
Notice that, thanks to \Cref{struct:lpz}, $\Xi$ is strongly convex, and that we have, for all $\mu\in\E$,
\begin{equation*}
\subgrad\Xi\mu=\subgrad{\Fen f{}}{\mu_1}\times\left\{\mu_2\right\}.
\end{equation*}
Thus, for all $\hat\mu\in\E$, $\mu_1\in\dom{\subgrad{\Fen f{}}{}}$, $w_1\in\subgrad{\Fen f{}}{\mu_1}$, and $\mu_2\in\E_2$, we have
\begin{equation}
\xiBreg{\hat\mu}{\mu}{(w_1,\mu_2)}=\fstarBreg{\hat\mu_1}{\mu_1}{w_1}+\frac12\sqEnorm{\hat\mu_2-\mu_2}2.
\end{equation}
\end{remark}

\begin{remark}
The idea of using a Bregman distance derived from $f^*$ is reminiscent from the article \cite{Bach2015Duality}, in which a dual interpretation of the agnostic FWA is given as a mirror descent algorithm. The Bregman distance involved in the mirror descent is the one associated with the Fenchel conjugate of the cost function.
\end{remark}

The following lemma is required for the statement of our last numerical assumption. Its proof is deferred to Section~\ref{sec:proofs}.

\begin{restatable}[label=it:esp]{lemma}{lemexpectation}
Let $S\subset\E_1$ be compact. There exists a bounded linear operator $\mathfrak E\colon\descrM{}{S}\to\E_1$ such that, for all $m\in\descrM{}{S}$, $\mathfrak Em$ is the unique vector verifying
\begin{equation*}
\forall\mu\in\E_1,\quad\int_S\EInnerProd\mu v1dm(v)=\EInnerProd\mu{\mathfrak Em}1.
\end{equation*}
Moreover, if $m\in\prob{S}$, then $\mathfrak Em\in \cloconv S.$
\end{restatable}

\begin{remark}
The vector $\mathfrak Em$ can be interpreted as the $m$-average over $S$. We also point out that simple sets are compact.
\end{remark}

Our last numerical assumption is the following.

\begin{enumerate}[label = (\upshape N\arabic*), ref = \upshape N\arabic*, start = 4]
\item\label[assum]{assum:subpb} Let $S\subset \E_1$ be simple. We can compute efficiently a solution to the problem
\begin{problem}{}{pb:expl_dual_proj}
\minimize{m\in-\descrM+S}\, \tilde f\left(-\mathfrak Em+w_1^t,1-m(S)\right)
+\frac12\left(\sqEnorm{\tilde\mu_2}{2}-\sqdist{\tilde\mu_2}{Q}\right)
-\ell^t m(S),
\end{problem}
where $\tilde\mu_2=-A\,\mathfrak Em+m(S)b+\mu_2^t$.
\end{enumerate}

\begin{remark}
If $S$ is a finite set $\left\{s_j, j\in\IntInt1{J}\right\}$, then finding $m\in\descrM{}{S}$ consists in finding a vector $\left(m_j\right)_{j\in\IntInt1{J}}\in\R^{J}$, and in this case $\mathfrak Em=\sum_{j=1}^{J}m_js_j$.
\end{remark}

\begin{remark}
The availability of efficient methods for the resolution of \Cref{pb:optim_combi_conv,pb:expl_dual_proj}, involved in the numerical \cref{assum:simple_sets,assum:subpb}, heavily depends on the nature of $f$ and $Q$. Let us insist on the fact that they do not need the knowledge of $f^*$.
Concerning the resolution of \Cref{pb:optim_combi_conv}, we note that, if $Q$ is bounded and $K^{-1}$ is finite, then the problem amounts to minimizing a Lipschitz-continuous cost function over a convex hull, which can be done with the mirror descent algorithm \cite{beck2003mirror}.
\Cref{pb:expl_dual_proj} involves a smooth cost function to be minimized with simple constraints. Thus it can therefore easily be handled with the projected gradient method.
Let us also note that the complexity of the two problems possibly increases along the iterations as we have no a priori upper bounds on the cardinality of $K'$. We will further comment on this issue in the conclusion of the article. 
\end{remark}

We are finally in position to state our DLS algorithm, its statement is provided in \Cref{alg:dualExtLNN}. The main idea to obtain it is to apply \Cref{alg:ExtLNN_prun} to \cref{pb:now_prim}, which gives \Cref{alg:LNN_appli_formal} (where we use \eqref{eq:psi-we-use} as well as \Cref{rem:subgrad_and_breg} to rewrite expressions depending on $\psi$ and $\Xi$ in terms of $f^\ast$ and $Q$). Suitable manipulations, which we detail below in Subsection~\ref{subsec:equivalence}, allow one to prove that our DLS algorithm, \Cref{alg:dualExtLNN}, is an instance of \Cref{alg:LNN_appli_formal}. 

\begin{figure}[p!]
\begin{algorithm}[H]
\label[algo]{alg:dualExtLNN}
\caption{Dualized Level-Set method for \Cref{pb:int_DLS}}
\SetKwBlock{Oracle}{Oracle:}{}
\SetKwBlock{DualUpdate}{Dual update:}{}
\SetKwBlock{PrimalUpdate}{Primal update:}{}
\SetKwBlock{DualCandidate}{Dual candidate:}{}
\SetKwBlock{Pruning}{Pruning:}{}
\textbf{Require:} $\mu^0 \in \dom{\subgrad{\Fen f{}}{}}\times Q$ such that $\mu^0_1 + A^\ast \mu^0_2 \in \dom{\support{-K}{}}$, $w^0_1\in\subgrad{\Fen f{}}{\mu_1^0}$, $\lambda\in(0,1)$\;
Set $\bar h^{-1}=+\infty$, $\hat\mu^{-1}=\mu^0$, and $\bar\Delta=+\infty$\;
\For{$t=0,\ldots$}{
\emph{Available at iteration $t$:}
$\mu^t\in\E$, $\hat\mu^{t-1}\in\E$, $\hat x_1^t\in\E_1$, $w^t_1\in\E_1$, $\bar h^{t-1}\in\bar\R^+$, $\bar\Delta\in\R_+\cup\{+\infty\}$, $K^{t-1}\subset\E_1$\;
\Oracle{
Find $v_1^t\in-\argmin{v_1\in K}\EInnerProd{\mu_1^t+A^*\mu_2^t}{v_1}1$ and set $\tilde K^t=\cloconv{K^{t-1}\cup\left\{-v_1^t\right\}}$\;
}
\smallskip
\DualUpdate{
Set 
$\bar h^t=\min\left\{\bar h^{t-1},\EInnerProd{\mu_2^t}b2+\EInnerProd{\mu_1^t}{w_1^t}1+\EInnerProd{\mu_1^t+A^*\mu_2^t}{v_1^t}1-f\left(w_1^t\right)\right\}
$\;
\eIf
{$\bar h^t<\bar h^{t-1}$}
{Set $\hat\mu^t=\mu^t$}
{Set $\hat\mu^t=\hat\mu^{t-1}$}
}
\smallskip
\PrimalUpdate{
Take a solution $\hat x_1^t$ to:
$\displaystyle\minimize{x_1 \in \tilde{K}^t} f (x_1) + \sigma_Q(Ax_1 - b)$\;
Set $\underline h^t = -f\left(\hat x^t_1\right)-\support Q{A\hat x^t_1-b}$, $\Delta^t = \bar h^t-\underline h^t$, and $\ell^t = \lambda \bar h^t+(1-\lambda)\underline h^t$\;
}
\smallskip
\Pruning{
\eIf{
$\Delta^t<(1-\lambda)\bar\Delta$
}
{
Take $K^t\subset\cloconv{K^{t-1}\cup\left\{-v^t\right\}}$ simple and such that
$\hat x_1^t\in\cloconv{K^t}$\;
Set $\bar\Delta=\Delta^t$\;
}
{
Set $K^t = K^{t-1}\cup\left\{-v_1^t\right\}$\;
}
}
\smallskip
\DualCandidate{
Find $m^t\in -\descrM{+}{K^t}$ solution to \Cref{pb:expl_dual_proj} with $S = K^t$\; \vspace{1mm}
Set $w^{t+1}_1=\frac{w^t_1 - \mathfrak E m^t}{1 - m^t(K^t)}$, $\mu_1^{t+1}=\grad f{w^{t+1}_1}$, and $\mu_2^{t+1}=\proj{\mu_2^t+m^t(K^t)b-A\,\mathfrak Em^t}Q$\;
}
\smallskip
}
\end{algorithm}

\medskip

\begin{algorithm}[H]
\label[algo]{alg:LNN_appli_formal}
\caption{Extended Level-Set method for \Cref{pb:now_prim}}
\SetKwBlock{Oracle}{Oracle:}{}
\SetKwBlock{DualUpdate}{Dual update:}{}
\SetKwBlock{PrimalUpdate}{Primal update:}{}
\SetKwBlock{DualCandidate}{Dual candidate:}{}
\SetKwBlock{Pruning}{Pruning:}{}
\textbf{Require:} $\mu^0\in\dom{\support E{}}\cap\left(\dom{\subgrad{\Fen f{}}{}}\times Q\right)$, $w_1^0\in\subgrad{\Fen f{}}{\mu_1^0}$, $\lambda\in(0,1)$\;
Set $\bar h^{-1}=+\infty$ and $\bar\Delta = +\infty$\;
\For{$t=0,\ldots$}{
\emph{Available at iteration $t$:}
$\mu^t\in\E$, $w^t\in\E$, $\bar h^{t-1}\in\bar\R^+$, $\bar\Delta\in\R_+\cup\{+\infty\}$, $E^{t-1}\subset\E$\;
\Oracle{
Choose $v^t\in\subgrad{\support E{}}{\mu^t}$ and set $\tilde{E}^t= \cloconv{E^{t-1}\cup\left\{v^t\right\}}$\;
}
\DualUpdate{
Set $\bar h^t=\min\left\{\bar h^{t-1},\Fen f{\mu_1^t}+\EInnerProd{\mu_2^t}b2+\support E{\mu^t}\right\}$\;
}
\smallskip
\PrimalUpdate{
Take a solution $\hat x^t$ to:
$\displaystyle\minimize{(x_1,x_2) \in -\tilde{E}^t} f (x_1) + \sigma_Q(x_2 - b)$\;
Set
$\underline h^t=\displaystyle\inf_{\mu\in\E_1\times Q}\Fen f{\mu_1}+\EInnerProd{\mu_2}b2+\support{\tilde E^{t}}\mu$, $\Delta^t = \bar h^t-\underline h^t$, and $\ell^t = \lambda \bar h^t+(1-\lambda)\underline h^t$\;
}
\smallskip
\Pruning{
\emph{As in \Cref{alg:ExtLNN_prun}.}
}
\smallskip
\DualCandidate{
Set $Q^t = \{\psi+\support{E^t}{}\leq \ell^t\}$\;
Take $\mu^{t+1}$ as the solution to:
$\minimize{\mu\in Q^t} \fstarBreg{\mu_1}{\mu_1^t}{w_1^t}+\frac12\sqEnorm{\mu_2-\mu_2^t}2$\;
Take $w_1^{t+1}\in\subgrad{\Fen f{}}{\mu_1^{t+1}}$ such that
$(w_1^t-w_1^{t+1}, \mu_2^t -\mu_2^{t+1}) \in N_{Q^t}(\mu^{t+1})$\;
}
\smallskip
}
\end{algorithm}
\end{figure}

\subsection{Duality between the ELS and the DLS methods}
\label{subsec:equivalence}

We justify in this subsection the fact that our DLS method, \Cref{alg:dualExtLNN}, is an instance of the ELS method applied to the dual problem \Cref{pb:now_prim}, \Cref{alg:LNN_appli_formal}.

We first note that the requirements of \Cref{alg:dualExtLNN,alg:LNN_appli_formal} coincide, since $\dom{\support{E}{}} = \{\mu \in \E \mid \mu_1 + A^\ast \mu_2 \in \dom{\support{-K}{}}\}$. We next study the relations between corresponding steps of \Cref{alg:dualExtLNN,alg:LNN_appli_formal}.

\paragraph{Oracle}
From \Cref{coro:subgrad_oracle} and the definition of the adjoint operator, we get that, for all $\mu'\in\E$, if we set $\mu_1=\mu'_1+A^*\mu'_2$ and let $v_1$ be given by the LMO on $K$ at $\mu_1$, then $-(v_1,Av_1)\in\subgrad{\support E{}}{\mu'}.$
Thus, the LMO on $K$ allows us to find an element of $\subgrad{\support E{}}{\mu^t}$. Also, for all $t\in\N$, since $A$ is linear and bounded, we have $\tilde E^t=\left\{x\in\E\,|\,x_1\in -\tilde K^t, x_2=Ax_1\right\}$.

\paragraph{Dual update} The update of $\bar h^t$ in \Cref{alg:dualExtLNN} is justified using the \textbf{Oracle}, \Cref{lem:FenchelYoungEq}, and \Cref{rem:subgrad_support}.

\paragraph{Primal update} First notice that, for all $\mu\in\E$, we have
\begin{equation*}
\Fen{\left(f\oplus\support Q{\cdot-b}{}\right)}{\mu}=\Fen f{\mu_1}+\InnerProd{\mu_2}b+\indic Q{\mu_2}.
\end{equation*}
Let now $t\in\N$. Since $\tilde E^{t}$ is bounded, we have 
\begin{align*}
\inf_{\mu\in \E_1\times Q}\Fen f{\mu_1}+\EInnerProd{\mu_2}b2+\support{\tilde E^{t}}\mu & = -\inf_{x\in\E}f(x_1)+\support Q{x_2-b}+\indic{-\tilde E_t}x \\
&=-\inf_{x_1\in\E_1}f(x_1)+\support Q{Ax_1-b}+\indic{\tilde K^{t}}{x_1}.
\end{align*}
This justifies the updates of $\underline h^t$ and $\hat x^t$.

\paragraph{Pruning}
Notice that, for all $t\in\N$, the set $K^t$ is simple, and thus compact, and that the set $E^t$ can always be taken as
\begin{equation}\label{eq:def_Et}
E^t=\left\{x\in\E\,|\, x_1\in -K^t,\, x_2=Ax_1\right\}.
\end{equation}

\paragraph{Dual candidate}

We focus on the projection problem, that is,
\begin{problem}{}{pb:proj}
\minimize{\mu\in Q^t} \fstarBreg{\mu_1}{\mu_1^t}{w_1^t}+\frac12\sqEnorm{\mu_2-\mu_2^t}2.
\end{problem}
The next proposition, whose proof is provided in Section~\ref{sec:proofs}, collects the properties of this problem that allow one to justify that the \textbf{Dual candidate} step of \Cref{alg:dualExtLNN} is an instance of the corresponding step of \Cref{alg:LNN_appli_formal}.

\begin{prop}[label=prop:projection]
Let $t\in\N$ and $m^t$ be a solution to \Cref{pb:expl_dual_proj} with $S = K^t$. The following hold:
\begin{enumerate}[label = \itshape\roman*\,\upshape), ref = \itshape\roman*]
\item\label{res:duality} Up to a shift in its value, \Cref{pb:expl_dual_proj} is the dual of \Cref{pb:proj}.
\item\label{res:update} Set $w_1^{t+1}$ and $\mu^{t+1}$ as in \Cref{alg:dualExtLNN}, i.e., as
\begin{align*}
w^{t+1}_1&=\frac{w^t_1-\mathfrak Em^t}{1-m^t(K^t)}\\
\mu_1^{t+1}&=\grad f{w_1^{t+1}}\\
\mu_2^{t+1}&=\proj{\mu_2^t+m^t(K^t)b-A\,\mathfrak Em^t}Q.
\end{align*}
Then:
\begin{enumerate}[label = \itshape\alph*), ref = \itshape\alph*]
\item\label{res:update_mu} The solution to \Cref{pb:proj} is $\mu^{t+1}$.
\item\label{res:update_w} We have $w^{t+1}_1\in\subgrad{\Fen f{}}{\mu_1^{t+1}}$ and $(w_1^t-w_1^{t+1}, \mu_2^t -\mu_2^{t+1}) \in N_{Q^t}(\mu^{t+1})$.
Moreover, if \Cref{assum:subgrad} is satisfied, then we can bound $w_1^t$ uniformly in $t$.
\end{enumerate}
\end{enumerate}
\end{prop}

\subsection{Convergence analysis}

We now want to prove that, under our standing assumptions, \Cref{alg:dualExtLNN} converges. At the light of Subsection~\ref{subsec:equivalence}, it suffices to show that \Cref{alg:LNN_appli_formal} converges and, for that purpose, one is left to verify that the assumptions of the convergence theorem for \Cref{alg:ExtLNN_prun}, Theorem~\ref{th:cv_LMP}, are satisfied in our setting. We assume in the sequel that \Crefrange{struct:lpz}{assum:qualif} are verified.

We start by remarking that, thanks to \eqref{eq:psi-we-use}, \Cref{rem:subgrad_and_breg}, and the definition of $E$ in \Cref{pb:now_prim}, we immediately obtain the following result from \Cref{struct:mu_0} and a straightforward computation.

\begin{lemma}\label{lem:H1}
\Cref{assum:mu_0} is satisfied. 
\end{lemma}

We next turn to the verification of \Cref{assum:psi_growth}.

\begin{lemma}\label{lem:psi_growth}
\Cref{assum:psi_growth} is satisfied. 
\end{lemma}

\begin{proof}
We recall that we defined the set $E^{-1}$ as
\begin{equation*}
E^{-1}=\left\{x\in\E\,\big|\,x_1\in-K^{-1},\,x_2=Ax_1\right\},
\end{equation*}
and that we want to show that $E^{-1}$ is bounded and
\begin{equation*}
\converge{\psi(\mu)+\support{E^{-1}}\mu}{+\infty}{\Enorm\mu{}}{+\infty}.
\end{equation*}
Boundedness of $E^{-1}$ is immediate since $K^{-1}$ is bounded and $A$ is a bounded linear operator.

The qualification condition \labelcref{eq:qualif} implies that there exists $\varepsilon>0$ such that for all $\mu_2\in\E_2\setminus\{0\}$ we have
\begin{equation*}
\varepsilon\frac{\mu_2}{\Enorm{\mu_2}2}-b\in Q_2^\ominus-A\,\cloconv{K^{-1}}.
\end{equation*}
In other words, there exists $\varepsilon>0$ such that for all $\mu_2\in\E_2\setminus\{0\}$, there exists $x_{\mu_2}\in \cloconv{K^{-1}}$ and $z_{\mu_2}\in Q_2^\ominus$ such that 
\begin{equation}\label{eq:decomp_qualif1}
\varepsilon\frac{\mu_2}{\Enorm{\mu_2}2}-b=z_{\mu_2}-Ax_{\mu_2}.
\end{equation}
Now, we fix such an $\varepsilon>0$. \Cref{eq:decomp_qualif1} implies that, for all $\mu_2\in\E$, there exists $x_{\mu_2}\in \cloconv{K^{-1}}$ and $z_{\mu_2}\in Q_2^\ominus$ such that
\begin{equation}\label{eq:decomp_qualif2}
\varepsilon\Enorm{\mu_2}2-\EInnerProd{\mu_2}b2=\EInnerProd{\mu_2}{z_{\mu_2}}2-\EInnerProd{\mu_2}{Ax_{\mu_2}}2.
\end{equation}
Notice that \cref{eq:decomp_qualif2} also holds for $\mu_2=0$ (with arbitrary choices of $x_{\mu_2}\in \cloconv{K^{-1}}$ and $z_{\mu_2}\in Q_2^\ominus$).
Let then $\mu_1\in\E_1$ and $\mu_2\in Q$. We know that
\begin{itemize}
\item Using the decomposition $Q=Q_1+Q_2$, there exist $\mu_2^a\in Q_1$ and $\mu_2^b\in Q_2$ such that
\begin{equation}\label{decomp:Q}
\mu_2=\mu_2^a+\mu_2^b.
\end{equation}
\item There exist $x_{\mu_2}\in \cloconv{K^{-1}}$ and $z_{\mu_2}\in Q_2^\ominus$ satisfying \cref{eq:decomp_qualif2}.
\end{itemize}
We have
\begin{align}
\psi(\mu)+\support{E^{-1}}\mu&=\Fen f{\mu_1}+\EInnerProd{\mu_2}b2+\indic Q{\mu_2}+\support{E^{-1}}\mu\notag\\
&\geq \Fen f{\mu_1}+\EInnerProd{\mu_2}b2-\EInnerProd{\mu_1}{x_{\mu_2}}1-\EInnerProd{\mu_2}{Ax_{\mu_2}}2\label{l:1}\\
&\geq \Fen f{\mu_1}-C_{K^{-1}}\Enorm{\mu_1}1+\varepsilon\Enorm{\mu_2}2-\EInnerProd{\mu_2}{z_{\mu_2}}2\label{l:2}\\
&\geq \Fen f{\mu_1}-C_{K^{-1}}\Enorm{\mu_1}1+\varepsilon\Enorm{\mu_2}2-\EInnerProd{\mu_2^a}{z_{\mu_2}}2\label{l:3}\\
&\geq \Fen f{\mu_1}-C_{K^{-1}}\Enorm{\mu_1}1+\varepsilon\Enorm{\mu_2}2-C_{Q_1}\left(\varepsilon +C_{K^{-1}}\norm A+\Enorm b2\right)\label{l:4},
\end{align}
where
\begin{itemize}
\item \Cref{l:1} derives from the facts that $\mu_2\in Q$ and $(-x_{\mu_2},-Ax_{\mu_2})\in \cloconv{E^{-1}}$.
\item \Cref{l:2} is a consequence of the Cauchy--Schwarz inequality, of the fact that $\Enorm {x_{\mu_2}}2\leq C_{K^{-1}}$, and of \cref{eq:decomp_qualif2}.
\item \Cref{l:3} derives from the fact that $z_{\mu_2} \in Q_2^\ominus$, and thus $\EInnerProd{\mu_2^b}{z_{\mu_2}}2 \leq 0$ since $\mu_2^b \in Q_2$.
\item \Cref{l:4} is a consequence of the Cauchy--Schwarz inequality, of the fact that \cref{eq:decomp_qualif1} yields $\Enorm{z_{\mu_2}}2\leq \varepsilon+C_{K^{-1}}\norm A+\Enorm b2$, and of the fact that $\Enorm{\mu_2^a}2\leq C_{Q_1}$.
\end{itemize}
Finally, notice that the function $\mu_2\in\E_2\mapsto\varepsilon\Enorm{\mu_2}2$ is coercive and lower bounded, and so is the function $\mu_1\in\E_1\mapsto \Fen f{\mu_1}-C\Enorm{\mu_1}1\in\R$ as a consequence of \Cref{lem:f_star_coercive}. The expected result then derives from \Cref{lem:sum_coer}.
\end{proof}

We next use the inequality \eqref{l:4} from the proof of \Cref{lem:psi_growth} in order to verify \Cref{assum:xi}.

\begin{lemma}
\Cref{assum:xi} is satisfied.
\end{lemma}

\begin{proof}
Recall that $Q^{-1} = \{\psi + \support{E^{-1}}{} \leq \bar h^0\}$. Thus, by \eqref{l:4}, we obtain that $\Fen{f}{\mu_1}$ and $\Enorm{\mu_2}{2}$ are finite for every $\mu \in Q^{-1}$, yielding that $Q^{-1} \subset \dom{\Xi}$. In addition, it also follows from \eqref{l:4} that $\mu \mapsto \Fen{f}{\mu_1} - C_{K^{-1}} \Enorm{\mu_1}{1} + \varepsilon \Enorm{\mu_2}2$ is bounded over $Q^{-1}$, and we conclude thanks to the definition of $\Xi$ and the strong convexity of $\Fen{f}{}$. 
\end{proof}

\begin{lemma}
If \Cref{assum:subgrad} is satisfied, then \Cref{assum:w} is satisfied.
\end{lemma}

\begin{proof}
Recall that, for all $\mu\in\E$, $\subgrad\Xi\mu=\subgrad{\Fen f{}}{\mu_1}\times\{\mu_2\}$. Thus, we have $w^t=(w^t_1,\mu_2^t)\in\subgrad\Xi{\mu^t}$, for all $t\in\N$.
The only thing left to show is that there exists a constant $C_{\subgrad\Xi{}}>0$ such that, for all $t\in\N$, we have $\norm{w^t}\leq C_{\subgrad\Xi{}}$. This derives from the fact that $w_1^t$ is uniformly bounded in $t$, as stated in \Cref{prop:projection}.\ref{res:update}.\ref{res:update_w}, and that, for all $t \in \N$, we have $\mu^{t+1} \in Q^t \subset Q^{-1}$ and $Q^{-1}$ is bounded.
\end{proof}

\begin{lemma}\label{lem:H6}
If \Cref{assum:subgrad} is satisfied, then \Cref{assum:lpz-like} is satisfied.
\end{lemma}

\begin{proof}
We want to show there exists $C_\psi>0$ such that for all $t\in\N$ we have
\begin{equation*}
\abs*{\psi\left(\mu^{t+1}\right)-\psi\left(\mu^t\right)}\leq C\Enorm{\mu^{t+1}-\mu^t}{}.
\end{equation*}
Let $t\in\N$, we have
\begin{equation*}
\abs*{\psi(\mu^{t+1})-\psi(\mu^t)}=\abs*{\Fen f{\mu_1^{t+1}}-\Fen f{\mu_1^t}+\EInnerProd{\mu_2^{t+1}-\mu_2^t}b2}\leq\abs*{\Fen f{\mu_1^{t+1}}-\Fen f{\mu_1^t}}+\abs*{\EInnerProd{\mu_2^{t+1}-\mu_2^t}b2}.
\end{equation*} 
Recall that $w_1^{t+1} \in\subgrad {\Fen f{}}{\mu_1^{t+1}}$, and thus
\begin{equation*}
\Fen f{\mu_1^{t+1}}-\Fen f{\mu_1^t}\leq\EInnerProd{\mu_1^{t+1}-\mu_1^t}{w_1^{t+1}}{}\leq\Enorm{w_1^{t+1}}1\Enorm{\mu_1^{t+1}-\mu_1^t}1.
\end{equation*}
Likewise, we have
\begin{equation*}
\Fen f{\mu_1^t}-\Fen f{\mu_1^{t+1}}\leq \Enorm{w_1^{t}}1\Enorm{\mu_1^t-\mu_1^{t+1}}1.
\end{equation*}
This yields
\begin{equation*}
\abs*{\psi\left(\mu^{t+1}\right)-\psi\left(\mu^t\right)}\leq\left(\max\left\{\Enorm{w_1^{t+1}}1,\Enorm{w_1^t}1\right\}+\Enorm b2\right)\Enorm{\mu^{t+1}-\mu^t}{},
\end{equation*}
and the conclusion follows since $w^t_1$ in uniformly bounded in $t$, as stated in \Cref{prop:projection}.\ref{res:update}.\labelcref{res:update_w}.
\end{proof}

Gathering \Crefrange{lem:H1}{lem:H6} and combining with the discussion of Subsection~\ref{subsec:equivalence}, we obtain at once the following result.

\begin{restatable}[label=th:main]{theorem}{mainth}
Consider the Dualized Level-Set method from \Cref{alg:dualExtLNN} under \Cref{assum:subgrad} and \Crefrange{struct:lpz}{assum:qualif}. Then the primal-dual gap $\Delta^t$ converges to $0$ with a speed of order $1/\sqrt t$. Also, we have, for all $t\in\N$, $\hat x^t_2 = A\hat x^t_1$, and we have the same convergence speed for $f\left(\hat x^t_1\right)+\support Q{A\hat x^t_1-b}-V$ and for $\Fen f{\hat\mu_1}+\indic Q{\hat\mu_2}+\EInnerProd{\hat\mu_2}b2+\support E{\hat\mu}+V$.
\end{restatable}

\begin{remark}\label{rem:new_oracle}
\begin{itemize}[leftmargin=*]
\item We have kept \Cref{assum:subgrad} in its non-explicit formulation on purpose. It is of course satisfied if $K$ is bounded, since then $K$ is weakly compact and thus \Cref{pb:oracle_dls} has a solution for any $\mu_1$.
\item In the general context of the ELS algorithm, it is actually sufficient to require the existence of a constant $C_{\oracle}>0$ such that for any $\mu \in Q^{-1}$, the linear minimization oracle has a solution in $B(0,C_{\oracle})$.
In the more specific context of the DLS algorithm, we have $Q^{-1} \subset \dom{f^*} \times Q$. Moreover, $Q^{-1}$ is a bounded set. Therefore, \Cref{assum:subgrad,assum:oracle} can be replaced by the following one:
\begin{enumerate}[label=\upshape(A\arabic*), ref = \upshape A\arabic*, start=0]
\item\label[assum]{assum:new_oracle} For any $R>0$, there exists $C_{\oracle}$ such that for any $(\mu_1',\mu_2') \in (\dom{f^*} \times Q) \cap B(0,R)$, \Cref{pb:oracle_dls} has a solution $v\in B(0,C_{\oracle})$, when called with $\mu_1= \mu_1'+ A^*\mu_2'$.
\end{enumerate}
\end{itemize}
\end{remark}

\subsection{Extension of the Generalized Conditional Gradient}

A now classical extension of the Frank--Wolfe algorithm, called generalized Frank--Wolfe (or generalized conditional gradient method, see \cite{kunisch2022fast}), consists in linearizing only partially the cost function. The contribution of the cost function which is not linearized remains then in the oracle and replaces the characteristic function of the feasible set.

In this subsection, we show that our algorithm can handle a situation of this kind, thanks to a natural augmentation of the problem through a slack variable.
Let $\E_{1a}$ and $\E_2$ be Hilbert spaces. We aim at solving the following generalization of \Cref{pb:dep_DLS}:
\begin{problem}{}{pb:dep_ECGC}
\minimize{x_{1a}\in\E_{1a}}f_a\left(x_{1a}\right)+\support Q{A_ax_{1a}-b}+h\left(x_{1a}\right).
\end{problem}

\paragraph{Structural assumptions}
We make the same assumptions on $f_a$, $Q$, and $A_a$ as we made on $f$, $Q$, and $A$ in \Cref{subsec:framework_DLS}, and we assume that $h\in\plscc{\E_{1a}}$ and is supercoercive, i.e., $\lim_{\norm{x} \to +\infty} h(x)/ \norm{x} = + \infty$.

\paragraph{Qualification condition} We assume that we can find $K_a^{-1}\subset\dom h$ compact such that $h$ is bounded over $\cloconv{K_a^{-1}}$ and $-b\in\inter{Q_2^\ominus-A_a\,\cloconv{K_a^{-1}}}$.

\paragraph{Oracle} We assume that we have the following oracle: for all $\mu_{1a}\in\E_{1a}$, we can find a solution to the problem
\begin{equation*}
\minimize{v_{1a}\in\E_{1a}}\EInnerProd{\mu_{1a}}{v_{1a}}{1a}+h\left(v_{1a}\right).
\end{equation*}

\begin{remark}
This is indeed a more general framework, since we can take $h=\indic {K_a}{}$, in which case we make exactly the same \textbf{Structural assumptions} and make the same \textbf{Qualification condition} as in \Cref{subsec:framework_DLS}, and the \textbf{Oracle} is a LMO on $K_a$.
\end{remark}

Under these assumptions, we can rewrite our problem in the framework of \Cref{subsec:framework_DLS}. For this, we take $\E_1=\E_{1a}\times\R$, the function $f$ defined as, for all $x_1=(x_{1a},x_{1b})\in\E_1$
\begin{equation*}
f(x_1)=f_a(x_{1a})+x_{1b},
\end{equation*}
which is indeed convex with gradient $\beta$-Lipschitz continuous, the bounded linear operator $A$ defined as, for all $x_1\in\E_1$
\begin{equation*}
A\left(x_{1a},x_{1b}\right)=A_ax_{1a},
\end{equation*}
and the closed convex set $K=\epi h$.
In this context, $\Fen f{}$ is given by, for all $\mu_1=(\mu_{1a},\mu_{1b})\in\E_1$
\begin{equation*}
\Fen f{\mu_{1a},\mu_{1b}}=\Fen{f_a}{\mu_{1a}}+\indic{\{1\}}{\mu_{1b}}
\end{equation*}
and the adjoint operator of $A$ is the operator $A^*$ given, for all $\mu_2\in\E_2$, by $A^*\mu_2 = \left(A_a^*\,\mu_2,0\right)$.

Also, this new problem verifies the \textbf{Qualification condition} introduced in \Cref{subsec:framework_DLS}, with $K^{-1}= K_a^{-1} \times \{ M \}$, where $M$ denotes an upper bound of $h$ over $K_a^{-1}$. Clearly $K^{-1} \subset K$ is compact. Since $A\,\cloconv{K^{-1}}=A_a\,\cloconv{K_a^{-1}}$, we have $-b\in\inter{Q_2^\ominus-A\,\cloconv{K^{-1}}}$
and the problem is indeed qualified.

We now need to verify that we indeed have an \textbf{Oracle} for the rewritten problem. As was explained in \Cref{rem:new_oracle}, it is sufficient to verify \Cref{assum:new_oracle}. We fix an arbitrary constant $R>0$ and take $(\mu_1',\mu_2') \in (\dom{f^*} \times Q) \cap B(0,R)$. Therefore $\mu_1'= (\mu_{1a}',1)$, with $\mu_{1a} \in \dom{\Fen{f_a}{}}$. Let $\mu_{1a}= \mu_{1a}'+ A_a^* \mu_2'$.
We have
\begin{align*}
\bar v_1\in\argmin{v_1\in K}\EInnerProd{\mu_1'+A^*\mu_2'}{v_1}1
&\Leftrightarrow \bar v_1\in\argmin{(v_{1a},v_{1b})\in\epi h}\EInnerProd{\mu_{1a}'+A_a^*\,\mu_2'} {v_{1a}}{1a}+v_{1b}\\
&\Leftrightarrow \bar v_{1b}=h\left(\bar v_{1a}\right) \text{ and }\bar v_{1a}\in\argmin{v_{1a}\in\E_{1a}}\EInnerProd{\mu_{1a}} {v_{1a}}{1a}+h\left(v_{1a}\right).
\end{align*}
Then $\bar v_{1a}$ is given by our \textbf{Oracle}.
As a direct consequence of \Cref{lem:FenchelYoungEq} the above statements are equivalent to: $\bar{v}_{1a} \in \partial h^*(-\mu_{1a})$
and $\bar{v}_{1b}= \EInnerProd{-\mu_{1a}} {\bar{v}_{1a}}{1a} -h^*(-\mu_{1a})$. Since we have a bound on $(\mu_1',\mu_2')$, we also have one on $-\mu_{1a}$. Applying \cite[Propositions 14.15(ii) and 16.17(iii)]{bauschke2011convex}, we deduce that $\abs{h^*(-\mu_{1a})}$ and $\Enorm{\bar{v}_{1a}}{1a}$ are bounded by some constant independent of $(\mu_1',\mu_2')$, depending only on $R$. Then $\abs{\bar{v}_{1b}}$ is bounded (in the same sense). This concludes the verification of \Cref{assum:new_oracle}.

\section{Technical proofs}\label{sec:proofs}

In this section, we provide the proofs of \Cref{it:esp} and \Cref{prop:projection}. Both these results concern the dualization of the projection problem with respect to the Bregman distance associated with the function $\Xi$ from \eqref{eq:def-Xi}, \Cref{pb:proj}. Recall that, for $t \in \N$, \Cref{pb:proj} is
\begin{problem}{\labelcref*{pb:proj}}{pb:proj_rec}
\minimize{\mu\in Q^t}\fstarBreg{\mu_1}{\mu_1^t}{w_1^t}+\frac12\sqEnorm{\mu_2-\mu_2^t}2,
\end{problem}
and we denote here its value by $V^t$.

We want to write a problem equivalent to \Cref{pb:proj} which fits in the framework of the Fenchel--Rockafellar duality, i.e., which is of the form
\begin{problem}{}{pb:gen_prim}
\minimize{X \in \X_1} F(X)+G(LX)
\end{problem}where $L\colon\X_1\to\X_2$ is a bounded linear operator, $\X_1$ and $\X_2$ are Banach spaces, $F\in\plscc{\X_1}$, and $G\in\plscc{\X_2}$.
\smallskip

We recall that $B_{\Fen f{}}$ is defined, for all $\mu_1\in\E_1$, $\mu_1'\in\dom{\subgrad {\Fen f{}}{}}$, and $w_1\in\subgrad{\Fen f{}}{\mu_1'}$ as
\begin{equation*}
\fstarBreg{\mu_1}{\mu_1'}{w_1}=\Fen f{\mu_1}-\Fen f{\mu_1'}-\EInnerProd {\mu_1-\mu_1'}{w_1}1
\end{equation*}
and that, for all $\mu\in\E$, we have
\begin{equation*}
\mu\in Q^t\Leftrightarrow
\left\{
\begin{aligned}
& \mu_2\in Q,\\
& \forall v\in E^t, \Fen f{\mu_1}+\EInnerProd{\mu_2}b2+\EInnerProd\mu v{}-\ell^t\leq0.
\end{aligned}
\right.
\end{equation*}
Thus, \Cref{pb:proj} is equivalent to
\begin{alproblem}{}{pb:int_proj}
&\minimize{\mu\in\E_1\times Q}\Fen f{\mu_1}-\EInnerProd{\mu_1}{w_1^t}1+\frac12\sqEnorm{\mu_2}2-\EInnerProd{\mu_2}{\mu_2^t}2\\
&\text{s.t. }\forall v\in E^t, \Fen f{\mu_1}+\EInnerProd{\mu_2}b2+\EInnerProd\mu v{}-\ell^t\leq 0,\notag
\end{alproblem}
in the sense that 
the value of \Cref{pb:int_proj} is $V^t-\Fen f{\mu^t}+\EInnerProd{\mu_1^t}{w_1^t}1-\frac12\sqEnorm{\mu_2^t}2$ and that \Cref{pb:proj,pb:int_proj} have the same solution.
Since $v\in E^t$ \textit{iff} $v_1\in-K^t$ and $v_2=Av_1$, this last problem is itself equivalent to
\begin{alproblem}{}{pb:now_proj}
&\minimize{(\mu,z)\in\left(\E_1\times Q\right)\times\R} \ z-\EInnerProd{\mu_1}{w_1^t}1+\frac12\sqEnorm{\mu_2}2-\EInnerProd{\mu_2}{\mu_2^t}2+\indic{\epi{\Fen f{}}}{\mu_1,z}\\
&\text{s.t. }\forall v_1\in -K^t, z+\EInnerProd{\mu_2}b2+\EInnerProd{\mu_1+A^*\mu_2}{v_1}1-\ell^t\leq 0,\notag
\end{alproblem}
in the sense that they have same value, and that $(\mu,z)$ is the solution to \Cref{pb:now_proj} \textit{iff} $z=\Fen f{\mu_1}$ and $\mu$ is the solution to \Cref{pb:int_proj}.

Let us now show that \Cref{pb:now_proj} has the expected shape. 
Let $L^t\colon\E\times\R\to\descrC {}{K^t}{}{}$ be the bounded linear operator defined, for $(\mu, z) \in \E \times \R$, by
\begin{equation*}
L^t(\mu,z) = z+\EInnerProd{\mu_2}b2-\EInnerProd{\mu_1+A^*\mu_2}{}1.
\end{equation*}
For any $(\mu,z)$, $L^t(\mu,z)$ is indeed a continuous and bounded function, since it is affine and defined on a bounded set. Clearly $L^t$ is a linear operator, it is easy to verify that it is bounded. Next, given $(\mu,z)\in (\E_1 \times Q)\times\R$, we see that the constraint in \Cref{pb:now_proj} is satisfied \emph{iff}
$L^t(\mu,z)-\ell^t\in\descrC {}{K^t}{\R_-}{}$.
From this last point, we define $G^t\colon\descrC{}{K^t}{}{}\to\bar\R_+$ as, for all $\phi\in\descrC{}{K^t}{}{}$,
\begin{equation*}
G^t(\phi)=\indic{\descrC{}{K^t}{\R_-}{}}{\phi-\ell^t}
\end{equation*}
and $F^t\colon\E\times\R\to\bar\R$ as, for all $\mu\in\E$ and $z\in\R$,
\begin{equation}
\label{eq:Ft}
F^t(\mu,z) = z-\EInnerProd{\mu_1}{w_1^t}1+\indic{\epi{\Fen f{}}}{\mu_1,z}+\frac12\sqEnorm{\mu_2}2-\EInnerProd{\mu_2}{\mu_2^t}2+\indic Q{\mu_2},
\end{equation}
and we notice that $F^t\in\plscc{\E\times\R}$ and $G^t\in\plscc{\descrC{}{K^t}{}{}}$. Finally, \Cref{pb:now_proj} reads
\begin{problem}{}{pb:trueproj}
\minimize{(\mu,z)\in\E\times\R}F^t(\mu,z)+G^t(L^t(\mu,z)),
\end{problem}
which is under the form \eqref{pb:gen_prim}, as required. The dual problem to \Cref{pb:trueproj} is then
\begin{problem}{}{pb:dualproj}
\minimize{m\in\descrC{}{K^t}{}{*}}\Fen{F^t}{{L^t}^*m}+\Fen {G^t}{-m}
\end{problem}
where ${L^t}^*\colon\descrC{}{K^t}{}{*}\to\E\times\R$ is the dual operator of $L^t$.

The aim of what follows is to provide a more explicit expression of \Cref{pb:dualproj}. We start by identifying, in the following lemma, the set $\descrC{}{K^t}{}{*}$ over which we minimize as a set of measures. Such an identification is immediate since $K^t$ is compact.

\begin{lemma}[label=it:MEt]
We have $\descrC{}{K^t}{}{}=\descrC b{K^t}{}{}$ and $\descrC{}{K^t}{}{*}=\descrM{}{K^t}$. Moreover, $\descrM{}{K^t}$ is endowed with the total variation norm $\norm m = \abs m(K^t).$
\end{lemma}

Let us now prove \Cref{it:esp}.

\begin{proof}[Proof of \Cref{it:esp}]
Let $S$ be a compact subset of $\E$. Let us first show that there exists a bounded linear operator $\mathfrak E\colon\descrM{}S\to\E$ such that, for all $m\in\descrM{}S$, $\mathfrak Em$ is the unique vector verifying
\begin{equation*}
\forall\mu\in\E,\,\EInnerProd\mu{\mathfrak Em}{}=\int_S\EInnerProd\mu v{}dm(v).
\end{equation*}
First, we show that for all $m\in\descrM{}S$, there exists a unique vector $v_m$ such that, for all $\mu\in\E$, we have 
\begin{equation*}
\EInnerProd\mu{v_m}{}=\int_S\EInnerProd\mu v{}\,dm(v).
\end{equation*}
Since $S$ is compact, there exists $C_{S}\in\R$ such that $S\subset B(0,C_S)$. Let $m\in\descrM{}S$. We define $I_m\colon\E\to\bar\R$ as, for all $\mu\in\E$,
\begin{equation*}
I_m(\mu)=\int_S\EInnerProd\mu v{}\,dm(v).
\end{equation*}
Since for all $\mu\in\E$, $\abs*{I_m(\mu)}\leq C_S\norm m\Enorm\mu{}$ and for all $v\in\E$, $\mu\mapsto\EInnerProd\mu v{}$ is linear, then $I_m$ is linear and continuous. Thus, using Riesz's representation theorem, there exists a unique $v_m\in\E$ such that for all $\mu\in\E$, $I_m(\mu)=\EInnerProd\mu{v_m}{}.$ Furthermore, $m\in\descrM{}S\mapsto v_m$ is also linear. Let then $\mathfrak E\colon\descrM{}S\to\E$ be defined as, for all $m\in\descrM{}S$, $\mathfrak Em=v_m.$ Then $\Enorm{\mathfrak Em}{}\leq C_S\norm m$ and thus $\mathfrak E$ is continuous. 

It remains to show that, for a given $m\in\prob{S}$, we have $\mathfrak Em\in \cloconv S$.
It suffices for this to show that $\indic{\cloconv S}{v_m}\leq 0$. We have
\begin{equation*}
\indic{\cloconv S}{v_m}=\Fen{\support S{}}{v_m}=\sup_{\mu\in\E}\,\EInnerProd\mu{v_m}{}-\support S\mu.
\end{equation*}
For $\mu\in\E$, we have
\begin{equation*}
\EInnerProd\mu{v_m}{}-\support S\mu=\int_S\left(\EInnerProd\mu v{}-\support S\mu\right)dm(v)
\end{equation*}
by definition of $v_m$ and using the linearity of the integral and the fact that $m(S)=1$. Moreover, for all $v\in S$, we have
\begin{equation*}
\EInnerProd\mu v{}-\support S\mu\leq \indic Sv =0
\end{equation*}
using inequality \eqref{eq:fenchel_young}. Thus, by positivity of the integral, we have
$\EInnerProd\mu{v_m}{}-\support S\mu\leq 0$.
Since this holds for any $\mu\in\E$, we have $\indic{\cloconv{S}}{v_m}\leq0$, which is the required result.
\end{proof}

We next turn to the question of providing explicit expressions for the functions $F^{t^*}$ and $G^{t^*}$ and the linear operator $L^{t^*}$ appearing in \Cref{pb:dualproj}.

\begin{lemma}[label=lem:3.7] The functions $F^{t^*}$ and $G^{t^*}$ and the linear operator $L^{t^*}$ are as follows:
\begin{enumerate}[label=\itshape\roman*\,\upshape)]
\item\label{it:Ft_star} The function $\Fen{F^t}{}\colon\E\times\R\to\bar\R^+$ is given, for all $(x,s)\in\E\times\R$, by
\begin{equation*}
\Fen{F^t}{x,s}=\tilde f(x_1+w_1^t,1-s)+\frac12\left(\sqEnorm{x_2+\mu_2^t}2-\sqdist{x_2+\mu_2^t}Q\right).
\end{equation*}
\item\label{it:Gt_star} The function $\Fen{G^t}{}\colon\descrM{}{K^t}\to\bar\R^+$ is given, for all $m\in\descrM{}{K^t}$, by
\begin{equation*}
\Fen{G^t}m=\ell^tm(K^t)+\indic{\descrM+{K^t}}m.
\end{equation*}
\item\label{it:Lt_star} The operator ${L^t}^*$ is given, for all $m\in\descrM{}{K^t}$, by
\begin{equation*}
{L^t}^*m=\left(\left(-\mathfrak Em, - A\mathfrak Em + m(K^t) b\right), m(K^t)\right)
\end{equation*}
\end{enumerate}
\end{lemma}

\begin{proof}
\begin{enumerate}[label=\itshape\roman*\,\upshape), leftmargin=0pt, itemindent={\parindent+\widthof{\itshape iii\,\upshape) }}]
\item Let $x\in\E$ and $s\in\R$. We have
\begin{align*}
\Fen{F^t}{x,s}&=\sup_{(\mu,z)\in\E\times\R}\InnerProd{(\mu,z)}{(x,s)}-F^t(\mu,s)\\
&= \sup_{(\mu_1,z)\in\epi{\Fen f{}}}\EInnerProd{\mu_1}{x_1+w_1^t}1+z(s-1)+\sup_{\mu_2\in Q}\EInnerProd{\mu_2}{x_2+\mu_2^t}2-\frac12\sqEnorm{\mu_2}2.
\end{align*}
By \Cref{lem:tildef}, it holds that
\begin{equation*}
\sup_{(\mu_1,z)\in\epi{\Fen f{}}}\EInnerProd{\mu_1}{x_1+w_1^t}1+z(s-1)=\tilde f\left(x_1+w_1^t,1-s\right).
\end{equation*}
Then, using \cite[Example~13.5]{bauschke2011convex}, we have
\begin{equation*}
\sup_{\mu_2\in Q}\EInnerProd{\mu_2}{x_2+\mu_2^t}2-\frac12\sqEnorm{\mu_2}2=\frac12\left(\sqEnorm{x_2+\mu_2^t}2-\sqdist{x_2+\mu_2^t}Q\right).
\end{equation*} 

\item Let $m\in\descrM{}{K^t}$. We have
\begin{align*}
\Fen{G^t}m&
=\sup_{\phi\in\descrC{}{K^t}{}{}}\int_{K^t}\phi\,dm-\indic{\descrC{}{K^t}{\R_-}{}}{\phi-\ell^t}
=\sup_{\phi\in\descrC{}{K^t}{\R_-}{}}\int_{K^t}(\phi+\ell^t)\,dm\\
&=\ell^tm(K^t)+\sup_{\phi\in\descrC{}{K^t}{\R_-}{}}\int_{K^t}\phi\,dm
\ = \ \ell^tm(K^t)+\indic{\descrM+{K^t}}{m}.
\end{align*}

\item The operator ${L^t}^*$ is characterized by the relation
\begin{equation*}
\forall \mu\in\E,\forall z\in\R,\forall m\in\descrM{}{K^t},\,\int_{K^t}L^t(\mu,z)(v)\,dm(v)=\InnerProd{(\mu,z)}{{L^t}^*m}.
\end{equation*}
Using \Cref{it:esp}, we have, for all $\mu\in\E$, $z\in\R$, and $m\in\descrM{}{K^t}$
\begin{align*}
\int_{K^t}L^t(\mu,z)(v)\,dm(v)&=z\,m(K^t)+\EInnerProd{\mu_2}{m(K^t)b}2 - \int_{K^t}\EInnerProd{\mu_1 + A^\ast \mu_2}{v}{1}\,dm(v)\\
&=\InnerProd{(\mu_1,\mu_2,z)}{(-\mathfrak Em, - A \mathfrak Em + m(K^t)b,m(K^t))}.
\end{align*}
\end{enumerate}
\noindent This concludes the proof.
\end{proof}

Now that we have computed $F^{t^*}$ and $G^{t^*}$, we compute in the next lemma their subgradients.

\begin{lemma}[label=lem:subgrad_FtGtstar]
\begin{enumerate}[label=\itshape\roman*\,\upshape), widest*=2, leftmargin=*]
\item The function $\Fen{F^t}{}$ is Fréchet-differentiable over $\E\times\R_-$ with continuous gradient, and we have, for all $(x,s)\in\E\times\R_-$,
\begin{equation*}
\grad{\Fen{F^t}{}}{x,s} = \Big( \grad{f}{y_1},\,\proj{x_2+\mu_2^t}Q,\,-f\left( y_1 \right)+\EInnerProd{\grad f{y_1}}{y_1}1\Big),
\quad
\text{where } y_1= \frac{x_1+w_1^t}{1-s}.
\end{equation*}
\item For all $m\in\descrM+{K^t}$
\begin{equation*}
\subgrad{\Fen{G^t}{}}{m}=\ell^t\mathds 1_{K^t}+\left\{\phi\in\descrC{}{K^t}{\R_-}{}\,\big|\,\int_{K^t} \phi\,dm=0\right\},
\end{equation*}
where $\mathds 1_{K^t} \in \mathcal C(K^t)$ is the function constantly equal to $1$ in $K^t$, and, for all $m\in\descrM{}{K^t}\setminus\descrM+{K^t}$, $\subgrad{\Fen{G^t}{}}{m}=\varnothing.$
\end{enumerate}
\end{lemma}

\begin{proof}
\begin{enumerate}[label=\itshape\roman*\,\upshape), leftmargin=0pt, itemindent={\parindent+\widthof{\itshape iii\,\upshape) }}]
\item The conclusion follows from the following facts:
\begin{itemize}
\item The function $\tilde f$ is differentiable with continuous gradient over $\E\times\R_+^*$, with its formula given in \labelcref{eq:tildef}.
\item If $s\leq0$, then $1-s>0$.
\item Using \cite[Corollary~12.30]{bauschke2011convex}, we have $\grad{\sqdist{}Q}{}=2(\operatorname{Id}-\proj{}Q)$, which is continuous.
\end{itemize}

\item We denote by $\psi^t\colon\descrM{}{K^t}\to\R$ the function defined for all $m\in\descrM{}{K^t}$ as $\psi^t(m)=\ell^tm(K^t)$,
so that $\Fen{G^t}{}=\psi^t+\indic{\descrM+{K^t}}{}$.
Notice that $\psi^t$ is linear, and thus convex, and continuous, and that $\indic{\descrM+{K^t}}{}\in\plscc{\descrM{}{K^t}}$. Thus, we have, for all $m\in\descrM{}{K^t}$,
\begin{equation*}
\subgrad{\Fen{G^t}{}}m=\subgrad{\psi^t}m+\subgrad{\indic{\descrM+{K^t}}{}}m.
\end{equation*}
Clearly, this implies that, for all $m\in\descrM{}{K^t}\setminus\descrM+{K^t}$, we have $\subgrad{\Fen{G^t}{}}{m}=\varnothing.$ Now, let $m\in\descrM+{K^t}$.
We have
\begin{equation*}
\psi^t(m)=\int_{K^t}\ell^t\mathds1_{K^t}\,dm
\end{equation*}
and thus, $\subgrad{\psi^t}m=\left\{\ell^t\mathds1_{K^t}\right\}$.

We are thus left to compute $\subgrad{\indic{\descrM+{K^t}}{}}m$ for $m \in \descrM+{K^t}$. By definition, we have
\begin{equation*}
\subgrad{\indic{\descrM+{K^t}}{}}m=\left\{\phi\in\descrC{}{K^t}{}{}\,\big|\,\forall \bar m\in\descrM+{K^t},\int_{K^t}\phi(d\bar m-dm)\leq0\right\}.
\end{equation*} Let $\phi\in\descrC{}{K^t}{}{}$. We consider three cases.
\begin{enumerate}[label=\alph*)]
\item There exists $\bar v_1\in K^t$ such that $\phi(\bar v_1)>0$. To show that $\phi\not\in\subgrad{\indic{\descrM+{K^t}}{}}m$, we have to find $\bar m\in\descrM+{K^t}$ such that $\int_{K^t}\phi(d\bar m-dm)>0.$ We set $\bar m=m+\delta_{\bar v_1}$. Then clearly $\bar m\in\descrM+{K^t}$, and we have
\begin{equation*}
\int_{K^t}\phi(d\bar m-dm)=\phi(\bar v)>0.
\end{equation*}
And thus, $\phi\not\in\subgrad{\indic{\descrM+{K^t}}{}}m.$
\item We have $\phi\in\descrC{}{K^t}{\R_-}{}$ and $\int_{K^t}\phi\,dm<0.$ Set $\bar m=\frac12m$. Then $\bar m\in\descrM+{K^t}$ and 
\begin{equation*}
\int_{K^t}\phi(d\bar m-dm)=-\frac12\int_{K^t}\phi\,dm>0.
\end{equation*}
Thus, $\phi\not\in\subgrad{\indic{\descrM+{K^t}}{}}m.$
\item We have $\phi\in\descrC{}{K^t}{\R_-}{}$ and $\int_{K^t}\phi\,dm=0.$ Let $\bar m\in\descrM+{K^t}$. Then
\begin{equation*}
\int_{K^t}\phi(d\bar m-dm)=\int_{K^t}\phi\,d\bar m\leq0.
\end{equation*}
Thus, $\phi\in\subgrad{\indic{\descrM+{K^t}}{}}m.$
\end{enumerate}
Since those three cases make a partition of $\descrC{}{K^t}{}{}$, the result follows. \qedhere
\end{enumerate}
\end{proof}

Our next result proves that there is strong duality between between \Cref{pb:trueproj,pb:dualproj} and provides optimality conditions for these problems.

\begin{prop}[label=th:opt_cond]
There exists $(\bar\mu,\bar z)\in\E\times\R$ such that $F^t(\bar\mu,\bar z)<+\infty$ and $G^t$ is continuous at $L^t(\bar\mu,\bar z)$. Thus, \Cref{pb:trueproj,pb:dualproj} have opposite values and, for all $(\mu,z,m)\in\E\times\R\times\descrM{}{K^t}$, the following are equivalent:
\begin{enumerate}[label=\itshape\roman*\upshape\,), ref = \itshape\roman*]
\item $(\mu,z)$ is solution to \Cref{pb:trueproj} and $m$ is solution to \Cref{pb:dualproj}.
\item ${L^t}^*m\in\subgrad{F^t}{\mu,z}$ and $-m\in\subgrad{G^t}{L^t(\mu,z)}.$
\item\label{th4.8line:iii} $(\mu,z)=\grad{\Fen{F^t}{}}{{L^t}^*m}$ and $L^t(\mu,z)\in\subgrad{\Fen{G^t}{}}{-m}$.
\end{enumerate}
\end{prop}

\begin{proof}
Notice that $\psi + \support{E^t}{} \in \plscc{\E}$ and it is coercive since it is lower bounded by $\psi + \support{E^t}{}$, which is coercive thanks to \Cref{lem:psi_growth}. Let then $\bar\mu\in\E$ be a solution to 
\begin{equation*}
\minimize{\mu\in\E} \psi(\mu) + \support{E^t}\mu
\end{equation*}
and set $\alpha = \psi(\bar\mu) + \support{E^t}{\bar\mu}$. Note that, since $E^t \subset \tilde E^t$ (where $\tilde E^t$ is the set defined in \Cref{alg:LNN_appli_formal}), we deduce that $\alpha \leq \underline h^t$ (where we use once again the notations of \Cref{alg:LNN_appli_formal}).

It follows from \eqref{eq:psi-we-use} that $\Fen{f}{\bar\mu_1} < +\infty$ and $\bar\mu_2 \in Q$, and we deduce from \eqref{eq:Ft} that $F^t(\bar\mu,\Fen f{\bar\mu_1})<+\infty$. Notice also that, for every $v \in K^t$, we have
\[
L^t(\bar\mu,\Fen f{\bar\mu_1})(v) = \psi(\bar\mu) - \EInnerProd{\bar\mu}{(v, A v)}{2} \leq \psi(\bar\mu) + \support{E^t}{\bar\mu} = \alpha \leq \underline h^t < \ell^t,
\]
and thus $G^t$ is continuous at $L^t(\bar\mu,\Fen f{\bar\mu_1})$.

We thus obtain that $0 \in \inter{L^t \dom{F^t} - \dom{G^t}}$, and the other conclusions of the proof follow from \Cref{th:FenchelRockafellar}, \Cref{coro:fenchel_rockafellar}, and the fact that $1-m(K^t)>0.$
\end{proof}

The aim of our next results is to show that, if $w^{t+1}$ and $\mu^{t+1}$ are defined as in \Cref{alg:dualExtLNN}, then they necessarily satisfy the properties required for the corresponding elements in \Cref{alg:LNN_appli_formal}. We start with the following property of $w^{t+1}$.

\begin{lemma}[label=lem:wt]
Let $m^t$ be a solution to \Cref{pb:dualproj} and set $w^{t+1}$ as in \Cref{alg:dualExtLNN}, i.e., as
\begin{equation*}
w^{t+1}=\frac{w^t-\mathfrak Em^t}{1-m^t(K^t)}.
\end{equation*}
Then $w^{t+1}\in\cloconv{\{w^t\}\cup K^t}$.
\end{lemma}

\begin{proof}
We recall that $m^t\in-\descrM{+}{K^t}$. There are two cases.
If $m^t\left(K^t\right) = 0$, then $w^{t+1} = w^t$. Otherwise, if $m^t\left(K^t\right) <0$ we set 
$\tilde m^t = \frac {m^t}{m^t\left(K^t\right)}\in\prob{K^t}$
so that $\mathfrak E\tilde m^t\in\cloconv{K^t}$. Then, using the linearity of $\mathfrak E$, we have
\begin{equation*}
w^{t+1} = \frac1{1-m^t\left(K^t\right)}w^t-\frac{m^t\left(K^t\right)}{1-m^t\left(K^t\right)}\mathfrak E\tilde m^t.
\end{equation*}
The result follows.
\end{proof}

Note that, if \Cref{assum:subgrad} is satisfied, then the sets $K^t$ are uniformly bounded in $t$. Hence, as an immediate consequence of \Cref{lem:wt}, we obtain the following result.

\begin{corollary}[label=coro:bound_wt]
Assume that \Cref{assum:subgrad} is satisfied. Then we can bound $w^t$ uniformly in $t$.
\end{corollary}

We next verify that the choices of $w_1^{t+1}$ and $\mu^{t+1}$ in \Cref{alg:dualExtLNN} satisfy the required conditions from \Cref{alg:LNN_appli_formal}.

\begin{lemma}[label=lem:ineq_wmu]
Let $t\in\N$ and consider the elements $\mu^t$, $\mu^{t+1}$, $w_1^t$, and $w_1^{t+1}$ defined as in \Cref{alg:dualExtLNN}. Define $Q^t$ as in \Cref{alg:LNN_appli_formal}. Then 
\begin{equation}
\label{eq:Qt-in-terms-of-algo-5}
Q^t = \{\mu \in \E_1 \times Q \mid \Fen{f}{\mu_1} + \EInnerProd{\mu_2}{b}2 - \EInnerProd{\mu_1 + A^\ast \mu_2}{v_1}{1} \leq \ell^t \text{ for all } v_1 \in \cloconv{K^t}\},
\end{equation}
$\mu^{t+1}$ is the solution to \Cref{pb:proj}, and
\begin{equation}
\label{eq:verify-algo-6}
(w_1^t-w_1^{t+1}, \mu_2^t -\mu_2^{t+1}) \in N_{Q^t}(\mu^{t+1}).
\end{equation}
\end{lemma}

\begin{proof}
First, \Cref{eq:Qt-in-terms-of-algo-5} derives from the definition of $Q^t$ in \Cref{alg:LNN_appli_formal} and from \cref{eq:def_Et}.

To prove the second statement, note that, thanks to \Cref{struct:lpz}, $f^\ast$ is strongly convex, and thus \Cref{pb:proj} admits a unique solution $\boldsymbol\mu$. Due to the discussion at the beginning of the section, $\boldsymbol\mu$ is solution to \Cref{pb:proj} if and only if $(\boldsymbol\mu, f^\ast(\boldsymbol\mu_1))$ is solution to \Cref{pb:trueproj}.

On the other hand, the element $m^t$ from \Cref{alg:dualExtLNN} is solution to \Cref{pb:expl_dual_proj} with $S = K^t$, which, thanks to \Cref{it:MEt,lem:3.7}, is equivalent to $m^t$ being a solution to \Cref{pb:dualproj}. Thus, by \Cref{th:opt_cond}, we necessarily have $(\boldsymbol\mu, f^\ast(\boldsymbol\mu_1)) = \nabla F^{t^\ast}(L^{t^\ast} m^t)$. Using the expression of $\nabla F^{t^\ast}$ from \Cref{lem:subgrad_FtGtstar} and that of $L^{t^\ast}$ from \Cref{lem:3.7}, we finally deduce that $\boldsymbol\mu = \mu^{t+1}$.

Let us now turn to the proof of \eqref{eq:verify-algo-6}. Note that, by definition of the normal cone, \eqref{eq:verify-algo-6} is equivalent to having
\begin{equation*}
\EInnerProd{w^t_1-w^{t+1}_1}{\mu_1-\mu_1^{t+1}}1+\EInnerProd{\mu_2^t-\mu_2^{t+1}}{\mu_2-\mu_2^{t+1}}2\leq 0
\end{equation*}
for every $\mu = (\mu_1, \mu_2) \in Q^t$.

For $\mu \in Q^t$, we set
\begin{equation*}
\Lambda^t(\mu)
= \EInnerProd{w^t_1-w^{t+1}_1}{\mu_1-\mu_1^{t+1}}1+\EInnerProd{\mu_2^t-\mu_2^{t+1}}{\mu_2-\mu_2^{t+1}}2.
\end{equation*}
First, we notice that
\begin{align*}
\MoveEqLeft \EInnerProd{\mu_2^t-\mu_2^{t+1}}{\mu_2-\mu_2^{t+1}}2 \\
& =\EInnerProd{\mu_2^t-A\,\mathfrak Em^t+m^t(K^t)b-\mu_2^{t+1}}{\mu_2-\mu_2^{t+1}}2+\EInnerProd{A\,\mathfrak Em^t-m^t(K^t)b}{\mu_2-\mu_2^{t+1}}2\\
& \leq \EInnerProd{A\,\mathfrak Em^t-m^t(K^t)b}{\mu_2-\mu_2^{t+1}}2,
\end{align*}
since $\mu_2^{t+1}=\proj{\mu_2^t-A\,\mathfrak Em^t+m(K^t)b}Q$.
Also notice that, by definition of $w_1^{t+1}$, we have
\begin{equation*}
w_1^t=\mathfrak Em^t+\left(1-m^t(K^t)\right)w_1^{t+1}
\end{equation*}
and thus
\begin{equation*}
w_1^t-w_1^{t+1}=\mathfrak Em^t-m^t(K^t)w_1^{t+1}
\end{equation*}
and
\begin{equation*}
\Lambda^t(\mu)\leq\EInnerProd{\mu-\mu^t}{\left(\mathfrak Em^t,A\,\mathfrak Em^t\right)}{}-m^t(K^t)\EInnerProd{\mu-\mu^{t+1}}{\left(w_1^{t+1},b\right)}{}.
\end{equation*}
Clearly, the result holds if $m^t=0_{\descrM{}{K^t}}$.

Now, assume that $m^t\neq 0_{\descrM{}{K^t}}$. We recall that, by definition of $\mu_1^{t+1}$ in \Cref{alg:dualExtLNN}, we have $w_1^{t+1}\in\subgrad{\Fen f{}}{\mu_1^{t+1}}$, and thus
\begin{equation*}
\EInnerProd{\mu_1-\mu_1^{t+1}}{w_1^{t+1}}1\leq \Fen f{\mu_1}-\Fen f{\mu_1^{t+1}}
\end{equation*}
which yields, using the fact that $-m^t\in\descrM+{K^t}$,
\begin{equation*}
\Lambda ^t(\mu)\leq\EInnerProd{\mu-\mu^{t+1}}{\left(\mathfrak Em^t,A\,\mathfrak Em^t\right)}{}-m^t(K^t)\left(\Fen f{\mu_1}-\Fen f{\mu_1^{t+1}}+\EInnerProd{\mu_2-\mu_2^{t+1}}b2\right).
\end{equation*}
Recalling that $m^t$ is solution to \Cref{pb:dualproj} and $(\mu^{t+1}, \Fen f{\mu_1^{t+1}})$ is solution to \Cref{pb:trueproj}, it follows from \Cref{th:opt_cond} that $L^t\left(\mu^{t+1},\Fen f{\mu_1^{t+1}}\right)\in\subgrad{\Fen{G^t}{}}{-m^t}$. Hence, from \Cref{lem:subgrad_FtGtstar}, we have
\begin{multline*}
\int_{K^t} \left[\Fen f{\mu_1^{t+1}}+\EInnerProd{\mu_2^{t+1}}b2-\ell^t-\EInnerProd{\mu_1^{t+1}+A^*\mu_2^{t+1}}{v_1}{}\right]\,dm^t(v_1) \\ = m^t(K^t)\left(\Fen f{\mu_1^{t+1}}+\EInnerProd{\mu_2^{t+1}}b2-\ell^t\right)-\EInnerProd{\mu_1^{t+1}+A^*\mu_2^{t+1}}{\mathfrak Em^t}{}=0,
\end{multline*}
and thus
\begin{equation*}
\Lambda^t(\mu)\leq \EInnerProd{\mu_1+A^*\mu_2}{\mathfrak Em^t}{}-m^t(K^t)\left(\Fen f{\mu_1}+\EInnerProd{\mu_2}b2-\ell^t\right).
\end{equation*}
Since $\mu\in Q^t$, it follows from \eqref{eq:Qt-in-terms-of-algo-5} that, for all $v_1\in \cloconv{K^t}$,
\begin{equation*}
\Lambda^t(\mu)\leq \EInnerProd{\mu_1+A^*\mu_2}{\mathfrak Em^t}{} - m^t(K^t)\EInnerProd{\mu_1+A^*\mu_2}{v_1}1.
\end{equation*}
Set $\displaystyle\bar v_1 = \frac{\mathfrak Em^t}{m^t(K^t)}$. We have $\bar v_1\in \cloconv{K^t}$ since $\displaystyle\frac{m^t}{m^t(K^t)}\in\prob{K^t}$ and $\mathfrak E$ is linear, and taking $v_1=\bar v_1$ in the above inequality yields $\Lambda^t(\mu)\leq0$, which concludes the proof.
\end{proof}

We finally collect the results of this section in order to provide a proof for \Cref{prop:projection}.

\begin{proof}[Proof of \Cref{prop:projection}]
\begin{enumerate}[label=\itshape\roman*\,\upshape), leftmargin=0pt, itemindent={\parindent+\widthof{\itshape iii\,\upshape) }}]
\item Thanks to \Cref{it:MEt,lem:3.7}, \Cref{pb:expl_dual_proj} coincides with \Cref{pb:dualproj}, and the latter is the dual of \Cref{pb:trueproj}, which is equivalent (up to a change in its value and a transformation in its variables, as discussed in the beginning of this section) to \Cref{pb:proj}.

\item Part \ref{res:update_mu} of the conclusion follows from \Cref{lem:ineq_wmu}, while part \ref{res:update_w} follows by combining the definition of $\mu_1^{t+1}$ in \Cref{alg:dualExtLNN}, \Cref{lem:FenchelYoungEq}, \Cref{coro:bound_wt}, and \Cref{lem:ineq_wmu}. \qedhere
\end{enumerate}
\end{proof}

\section{Numerical examples}\label{sec:numerics}

In this section, given $n\in\N^*$ and $p\in[1,+\infty]$, we denote by $B_p^n$ the closed unit ball in $\R^n$ for the $\ell^p$ norm, which is denoted by $\| \cdot \|_p$. We denote by $\mathcal{M}_{n,m}(\R)$ the set of real matrices of size $n \times m$. For square matrices, we write $\mathcal{M}_n(\R)$ instead of $\mathcal{M}_{n,n}(\R)$. The vector space of symmetric matrices of size $n$ is denoted
by $\mathcal S_n(\R)$ and the subset of those which are positive semidefinite is denoted by $\mathcal S_n^+(\R)$.

\subsection{A projection problem}
\label{subsec:test1}

We first test our algorithm in a simple problem taken from \cite{Silveti_Falls_2020}.
Let $p\in \{1, 2 \}$ and $A \in \mathcal{M}_{1,2}(\R) \backslash \{ 0 \}$. We aim at solving
\begin{problem}{}{expb:proj}
\minimize{x\in B^2_p}\frac12\sqnorm{x-y}{} \text{ s.t. } Ax=0.
\end{problem}
Notice that the feasible set of the problem is a segment, obtained as the intersection of a convex set and a line, whose extremities can be computed analytically. The projection of a point onto a segment being easy to solve, the problem can be solved analytically without difficulty.

\paragraph{Structure}
This problem fits in our framework if we take $\E_1 = \R^2$, $f = \frac12\sqnorm{\cdot-y}{}$, $\E_2 = \R$, $Q_1=\{0\}$, $Q_2 = \R$, $b = 0$, and $K = B^2_p$. For this choice, we indeed have $\sigma_Q= \iota_{\{ 0 \}}$. Notice that, since $K$ is compact, \Cref{assum:subgrad} is satisfied.

\paragraph{Qualification condition}
This problem is qualified. Indeed, if $\text{Ker}(A) = \R\times\{0\}$, then we can take $K^{-1} = \{(0,-1),(0,1)\}$, and otherwise we can take $K^{-1} =\{(-1,0),(1,0)\}.$

\paragraph{Oracle}
The LMO writes:
$\minimize{x\in B^2_p} \InnerProd\mu x$,
for a given $\mu \in \R^2$.
For $p=1$, the set $K= B^2_1$ is the convex hull of the four points $(0,1)$, $(0,-1)$, $(1,0)$, and $(-1,0)$, therefore, for any $\mu \in \R^2$, one of them is solution to the LMO, making its resolution easy.
For $p=2$, two cases must be considered. If $\mu= 0$, any point in $K$ is a solution. Otherwise, the unique solution is given by $-\mu/\| \mu \|_2$.

\paragraph{Numerical results}

We provide numerical results for a fixed value of $A$ and for $10^4$ iterations of the algorithm. For the two possible values of $p$, we consider two values of $y$, denoted $y_1$ and $y_2$ and chosen in such a way that, for $y= y_1$, the solution to \Cref{expb:proj} lies on the boundary of $K$ (i.e., on the unit sphere) and, for $y=y_2$, the solution lies in the interior of $K$.

Let us note that for $p=1$, the (primal) solution is obtained after finitely many iterations, since $K$ is the convex hull of finitely many points. We also obtain an optimal primal solution in a finite number of iterations if it lies in the interior of $K$ (for $y= y_2$). We do not expect in general to find the dual solution in finitely many iterations.

In \Crefrange{fig:ball_PDG_e_1}{fig:ball_PDG_i_2}, we show the evolution of the primal-dual gap $\Delta^t$ at each iteration, in log-log scale, for $\lambda\in \{0.05,0.1,1-\sqrt{2-\sqrt 2},0.5 \}$, for the pairs $(y,p)$ taken respectively as $(y_1,1)$, $(y_2,1)$, $(y_1,2)$, and $(y_2,2)$. For the pruning rule, we simply take $K^t= \left\{\hat x^t\right\}\cup K^{-1}$ at critical iterations.

We notice that, in all of these cases, the primal-dual gap show a numerical decrease toward $0$ with a speed of order $1/t$, with the curves for $\lambda=0.5$ decreasing slightly more slowly than those with the other choices of $\lambda$. We recall that our proof shows only a speed of order $1/\sqrt t$. In \Cref{fig:ball_cuts_e_2,fig:ball_cuts_i_2}, we show the number of cuts at each iteration for the same values of $\lambda$ and for the pairs $(y,p)$ taken respectively as $(y_1,2)$ and $(y_2,2).$ We do not show the number of cuts for $p=1$ since in that case it is bounded by 4 (the four extremal points of $K$, which contains $K^{-1}$).
We notice that the number of cuts stays lower for smaller values of $\lambda.$ This result was to be expected, since pruning steps should happen more often for lower values of $\lambda.$

\begin{figure}[hptb]
\centering
\begin{subfigure}{0.49\linewidth}
\centering
\includegraphics[width=\linewidth]{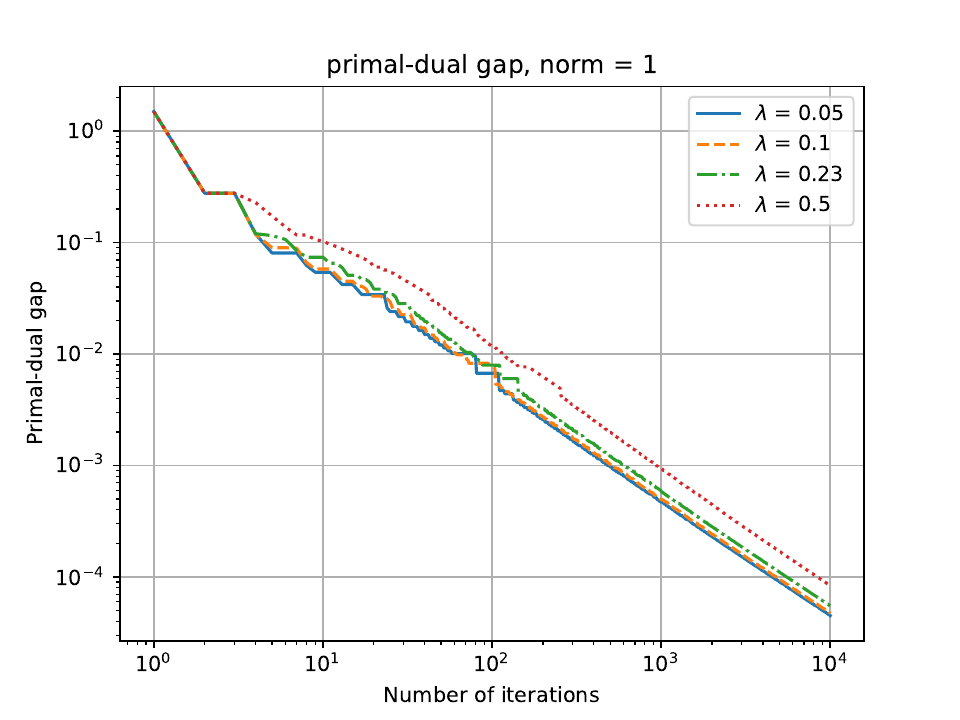}
\caption{$y=y_1$ and $p=1$}
\label{fig:ball_PDG_e_1}
\end{subfigure}
\hfill
\begin{subfigure}{0.49\linewidth}
\centering
\includegraphics[width=\linewidth]{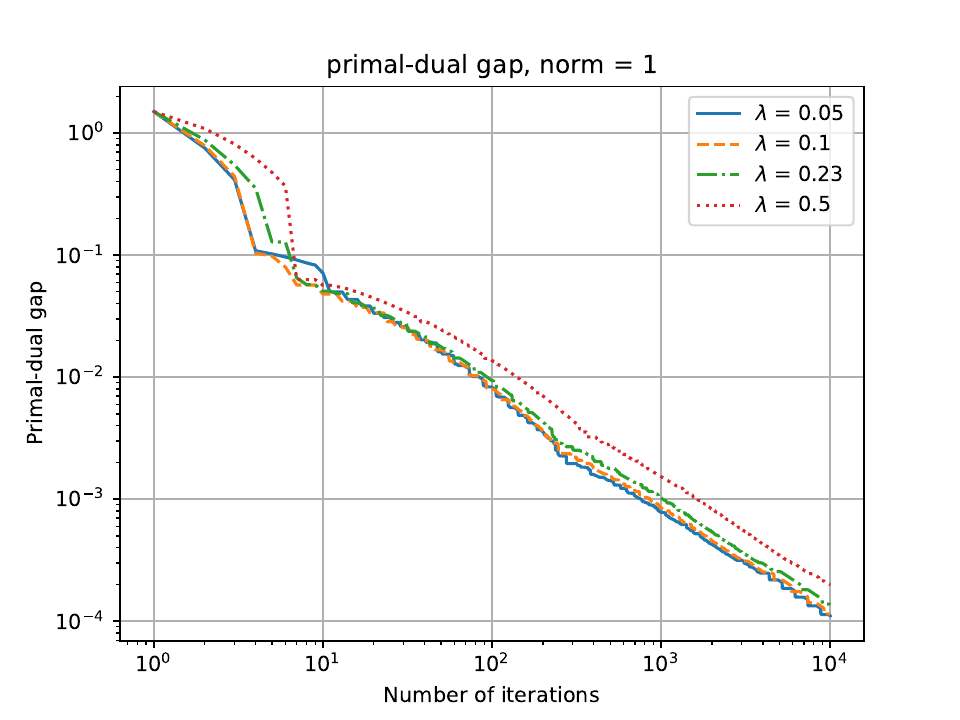}
\caption{$y=y_2$ and $p=1$}
\label{fig:ball_PDG_i_1}
\end{subfigure}
\hfill
\begin{subfigure}{0.49\linewidth}
\centering
\includegraphics[width=\linewidth]{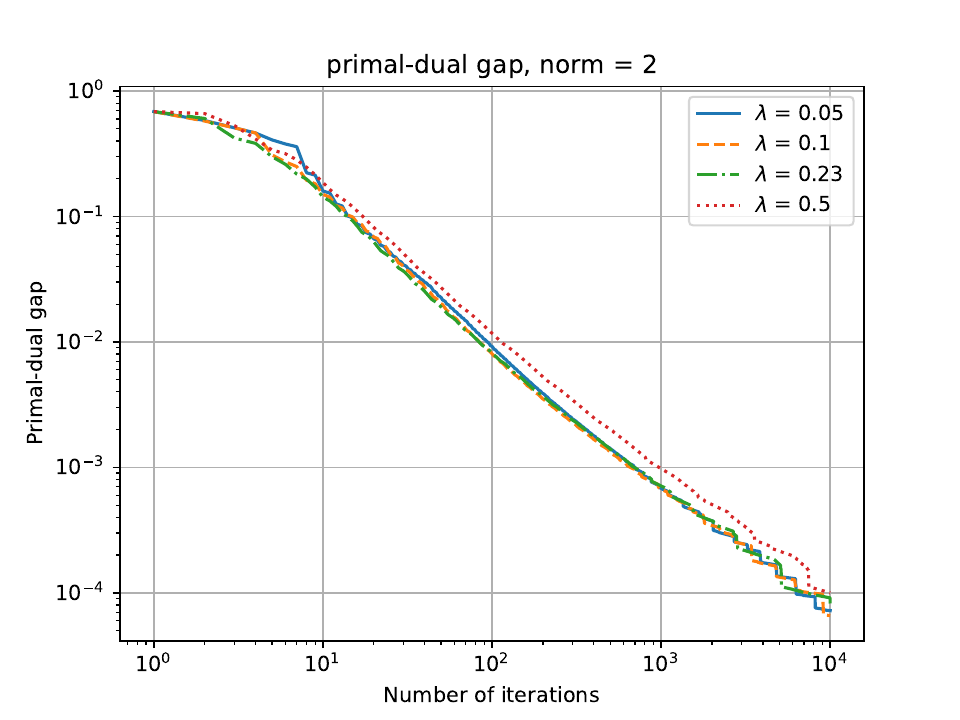}
\caption{$y=y_1$ and $p=2$}
\label{fig:ball_PDG_e_2}
\end{subfigure}
\begin{subfigure}{0.49\linewidth}
\centering
\includegraphics[width=\linewidth]{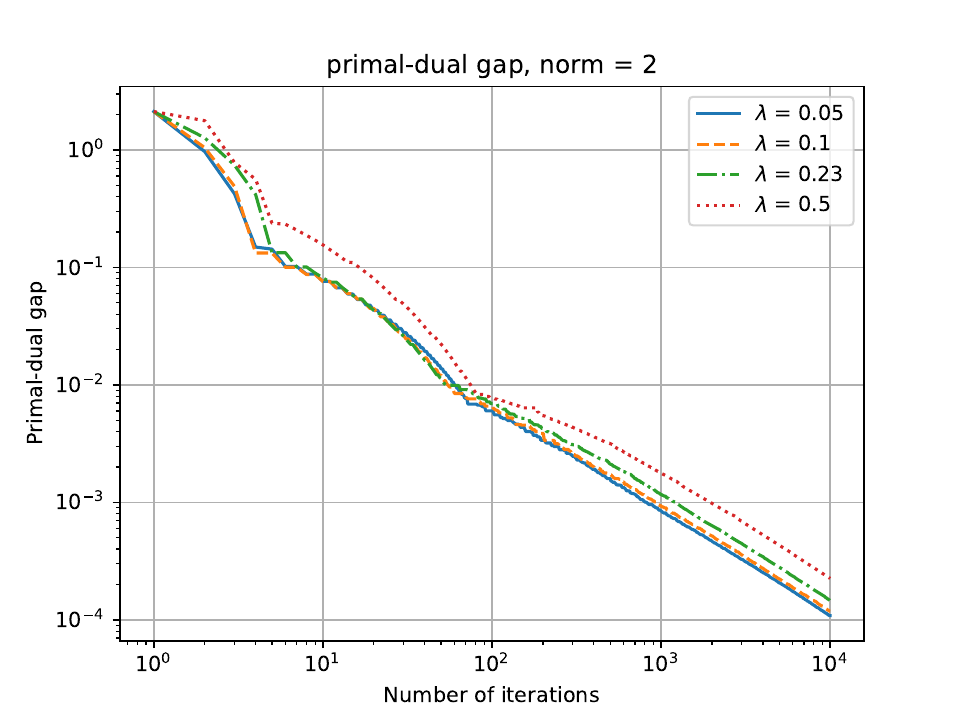}
\caption{$y=y_2$ and $p=2$}
\label{fig:ball_PDG_i_2}
\end{subfigure}
\caption{Primal-dual gap for 4 instances of \Cref{expb:proj}}
\label{fig:ball1}
\end{figure}

\begin{figure}[hptb]
\begin{subfigure}{0.49\linewidth}
\centering
\includegraphics[width=\linewidth]{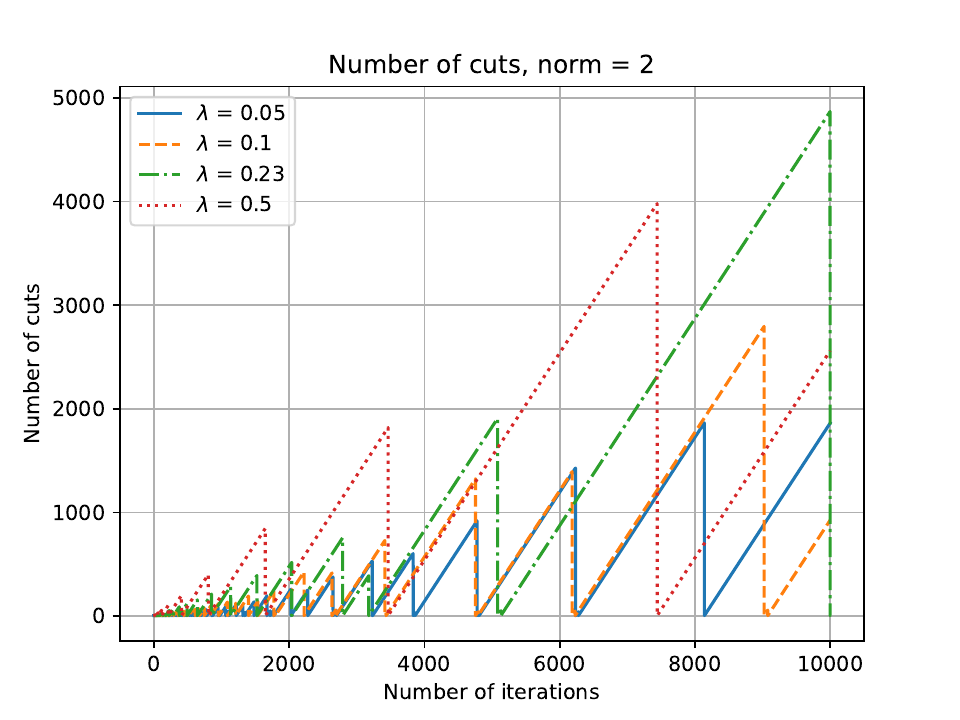}
\caption{$y=y_1$ and $p=2$}
\label{fig:ball_cuts_e_2}
\end{subfigure}
\begin{subfigure}{0.49\linewidth}
\centering
\includegraphics[width=\linewidth]{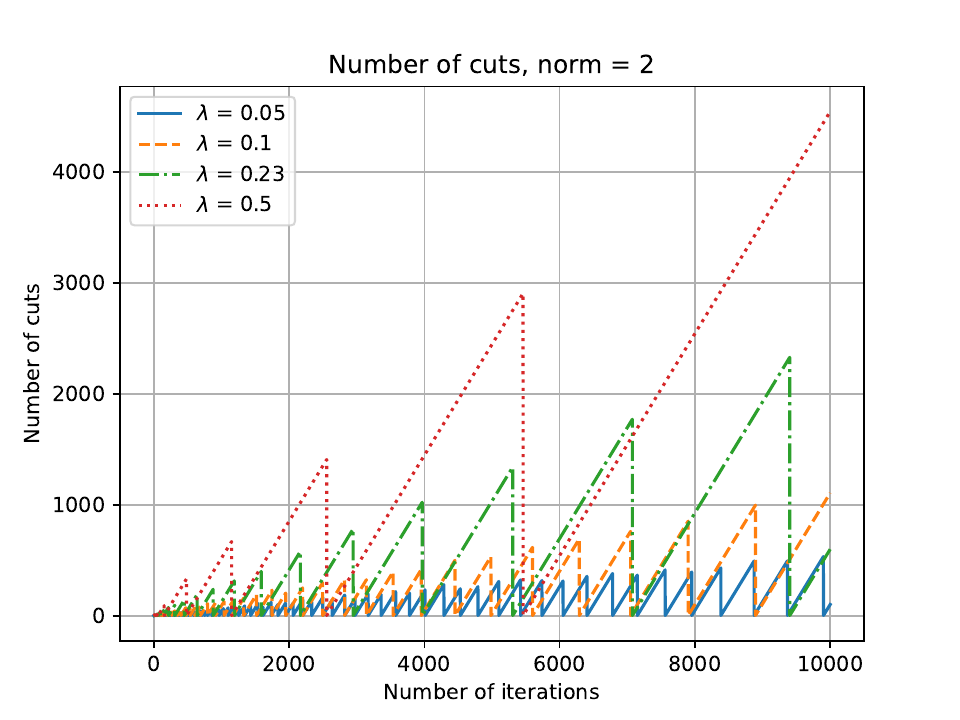}
\caption{$y=y_2$ and $p=2$}
\label{fig:ball_cuts_i_2}
\end{subfigure}
\caption{Evolution of the number of cuts for 2 instances of \Cref{expb:proj}}
\label{fig:ball2}
\end{figure}

\subsection{A semidefinite program}

We now test our algorithm with the following problem, extracted from \cite{yurtsever2019conditional}. Let $n\in\N^*$ and let $C\in\mathcal S_n(\R)$. We aim at solving
\begin{problem}{}{}
\minimize{X\in\mathcal S_n^+(\R)} \descrInnerProd CX{\M_n(\R)} \text{ s.t. diag}(X)= 1_{\R^n}
\end{problem}where $\descrInnerProd AB{\M_n(\R)} = \text{tr}(A^TB)$ is the canonical inner product.

\paragraph{Structure}
This problem fits in our framework if we take $\E_1 = \mathcal S_n(\R)$, $f = \descrInnerProd C{}{\M_n(\R)}$, $\E_2 = \R^n$, $A$ the linear operator which maps $X$ to its diagonal, $Q_1 = \{0\}$, $Q_2 = \R^n$, $b = 1_{\R^n}$ and $K = \mathcal S_n^+(\R)\cap \left\{\text{tr}\leq n+1\right\}$. Notice that, in this context, $A^*$ is the linear operator which maps a vector $Y$ to the diagonal matrix whose diagonal is $Y$. The definition of the set $K$ is dictated by the qualification condition and by the necessity to have an oracle.

\paragraph{Qualification condition}
This problem verifies the qualification condition by taking $K^{-1}$ as the convex hull of the null matrix and of the matrices $(n+1)E_{i,i},\, i\in\IntInt1n$, where $E_{i,j}$, $i,j\in\IntInt1n$ are the elementary matrices.

\paragraph{Oracle}
A direct application of the spectral theorem reveals that $K= (n+1) \, \conv{K'}$, where the set $K'$ is defined by
$K'= \{ 0 \} \cup \big\{ vv^T | \,\| v \|_2= 1 \big\}$. This allows to show the following lemma.

\begin{lemma}
Let $M\in\mathcal S_n(\R)$. Consider the problem
\begin{problem}{}{}
\minimize{V \in K}{\descrInnerProd MV{\M_n(\R)}}.
\end{problem}
Let $s$ denote the small eigenvalue of $M$. If $s \geq 0$, then the null matrix is a solution to the problem. Otherwise, if $s < 0$, let $\bar{v}$ be an eigenvector associated with $s$, of norm equal to 1. Then $(n+1) \bar{v} \bar{v}^T$ is a solution to the problem.
\end{lemma}

As in the previous subsection, since $K$ is bounded, \Cref{assum:subgrad} is satisfied.

\paragraph{Numerical results}
We set $n=10$ and take a random matrix $C\in\mathcal S_n(\R)$.
We then run the algorithm until the primal-dual gap $\Delta^t$ reaches $10^{-6}$. We use the same pruning rule as in \Cref{subsec:test1}. In \Cref{fig:mat_cuts}, we show the number of cuts at each iteration, respectively for $\lambda=0.05$ and $\lambda =0.7$. Although we do not show it, the number of cuts for the other values of $\lambda$ stays below $45$. We notice that, as we expected, the maximal number of cuts is lower for smaller values of $\lambda$.

\begin{figure}[H]
\centering
\includegraphics[width=0.49\linewidth]{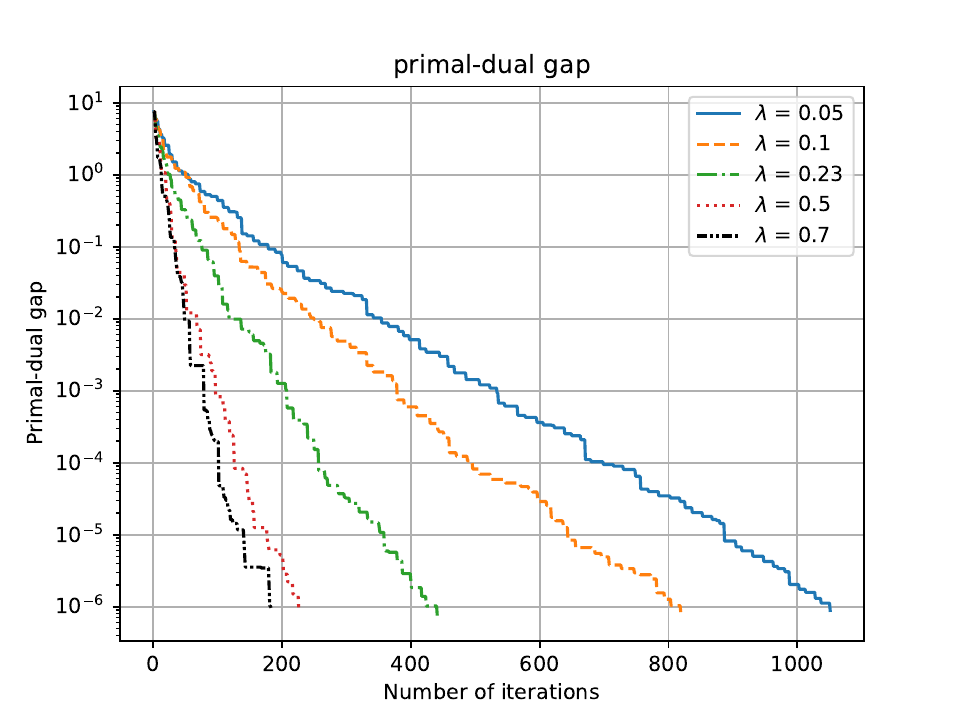}
\caption{Primal-dual gap}
\label{fig:mat_PDG}
\end{figure}

In \Cref{fig:mat_PDG}, we show the primal-dual gap at each iteration for $\lambda\in \big\{0.05,0.1,1-\sqrt{2-\sqrt 2},0.5,0.7\big\}$, in $y$-log scale. We can see that, numerically, our algorithm has a linear convergence speed for this problem. Also notice that, unlike the previous problem, the algorithm seems to converge faster for larger values of $\lambda.$ We do not know the reason for this behavior.

\begin{figure}[H]
\centering
\begin{subfigure}{0.49\linewidth}
\centering
\includegraphics[width=\linewidth]{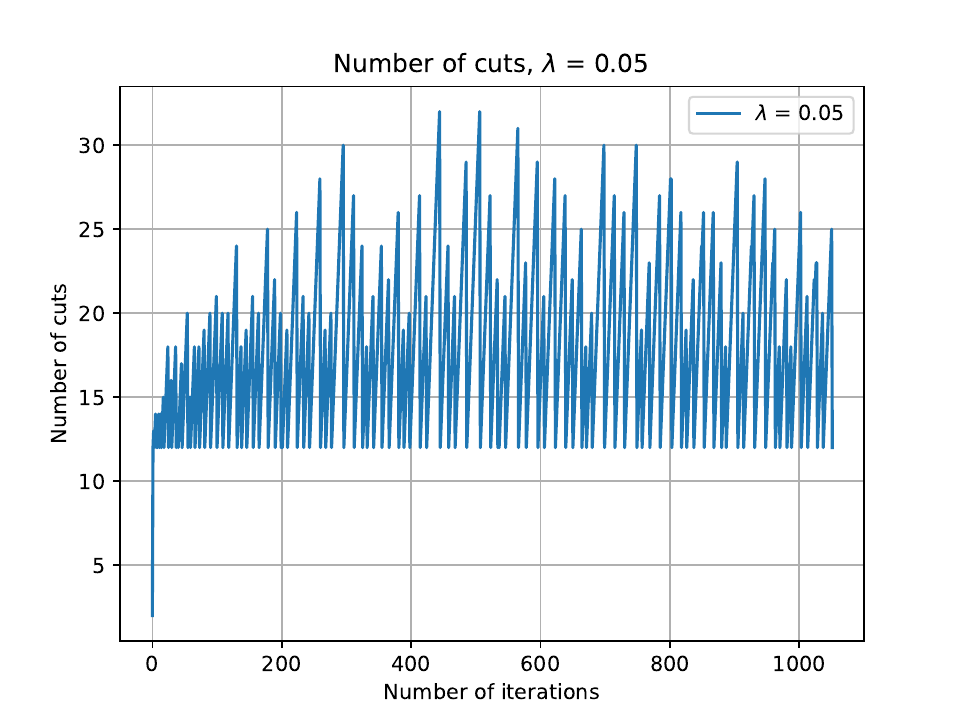}
\caption{$\lambda=0.05$}
\label{fig:mat_cuts_005}
\end{subfigure}
\begin{subfigure}{0.49\linewidth}
\centering
\includegraphics[width=\linewidth]{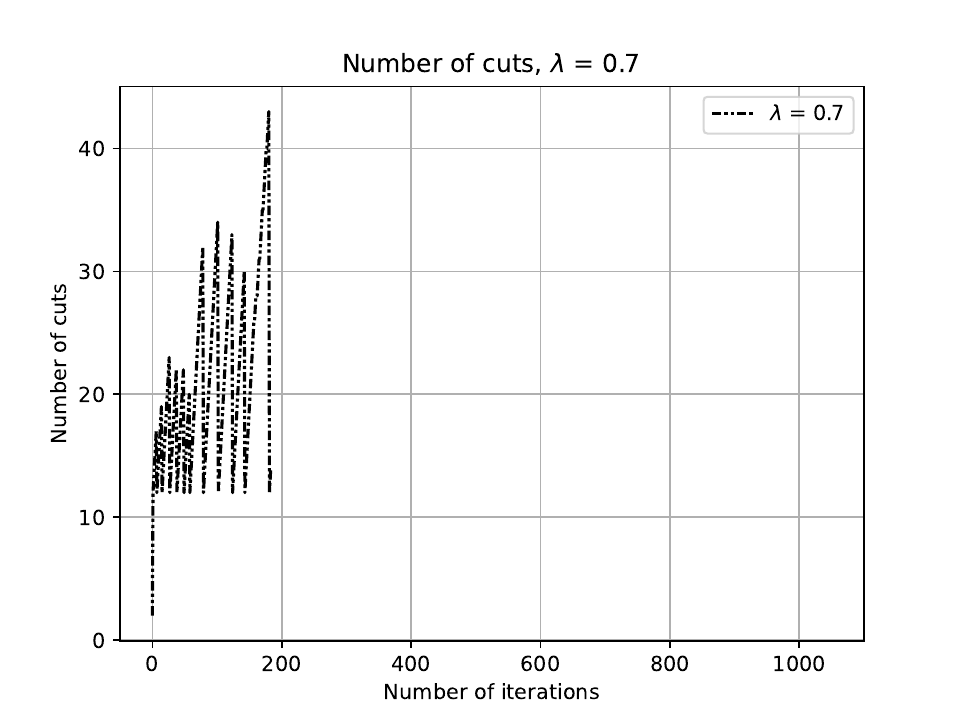}
\caption{$\lambda=0.7$}
\label{fig:mat_cuts_07}
\end{subfigure}
\hfill
\caption{Number of cuts at each iteration}
\label{fig:mat_cuts}
\end{figure}

\section{Conclusion and open problems}

In this article, we proposed a Dualized Level-Set method for problems of the form
\begin{equation*}
\minimize{x_1\in\E_1}f(x_1)+\support Q{Ax_1-b}+\indic{K}{x_1},
\end{equation*}
which we derived from an extension of the Level-Set method. We showed guarantees for the convergence of this algorithm. There are some improvement perspectives we can think of.
\begin{itemize}
\item Our experiments show better numerical convergence rates than we expected, which might be due to the specific form of these problems. This should not be surprising since some classical versions of the FWA are also known to have an improved convergence rate in some contexts, see for instance \cite[Section 2.2]{braun2023conditional}.

\item We assume that we are able to solve exactly our subproblems. An extension of our results could concern the situation where the problems are only solved up to a certain precision, as was done in \cite{Jaggi} for the FWA.

\item We chose to keep the parameter $\lambda$ fixed along the iterations. Our numerical results, in particular those done for the semidefinite program, show that the value of $\lambda$ may have a significant impact on the efficiency of the algorithm. Future improvements may therefore concern variants of the algorithm in which $\lambda$ is adapted along the algorithm. For instance, we think our proof could be adapted if we take $0<\lambda_{\min}\leq\lambda\leq\lambda_{\max}<1$, with changes of $\lambda$ occurring only at critical iterations.

\item The nonsmooth term in the cost function is restricted to be of the form $\support{Q}{Ax-b}$, with $Q$ of a specific shape. A possible extension may deal with the case of a general nonsmooth term $g$. From an algorithmic point of view, this would lead to a more general projection problem, whose dual (\Cref{pb:expl_dual_proj}) would involve the Moreau envelope of $g^*$. Yet some difficulties would arise in the extension of \Cref{lem:psi_growth}, whose proof heavily relies on the structure of the nonsmooth term.

\item Our numerical experience shows that pruning rules are unavoidable. Indeed, an implementation of the algorithm which retains all cuts computed at each iteration quickly becomes intractable. Yet the pruning rule that we proposed is not completely satisfactory in so far as we do not have a priori bounds on the number of cuts which need to be stored. As was seen on \Cref{fig:ball2}, the critical iterations may occur at diminishing frequency, leading to an accumulation of many cuts.
In \cite{Kiwiel}, the author develops a variant of the Level-Set method with a different pruning rule, for which the number of cuts to be stored is bounded by the dimension of the space. We expect that his method can be extended in the same manner as we extended the Level-Set method of \cite{lemarechal1995new}.
\end{itemize}

\bibliographystyle{plain} 
\bibliography{biblio.bib}

\begin{thebibliography}{10}

\bibitem{MAL-039}
Francis Bach.
\newblock Learning with submodular functions: A convex optimization
  perspective.
\newblock {\em Foundations and Trends® in Machine Learning}, 6(2-3):145--373,
  2013.

\bibitem{Bach2015Duality}
Francis Bach.
\newblock Duality between subgradient and conditional gradient methods.
\newblock {\em SIAM J. Optim.}, 25(1):115--129, 2015.

\bibitem{bauschke2011convex}
Heinz~H Bauschke and Patrick~L Combettes.
\newblock {\em Convex analysis and monotone operator theory in Hilbert spaces},
  volume 408.
\newblock Springer, 2011.

\bibitem{beck2003mirror}
Amir Beck and Marc Teboulle.
\newblock Mirror descent and nonlinear projected subgradient methods for convex
  optimization.
\newblock {\em Operations Research Letters}, 31(3):167--175, 2003.

\bibitem{bonnans2006numerical}
Joseph-Fr{\'e}d{\'e}ric Bonnans, Jean~Charles Gilbert, Claude Lemar{\'e}chal,
  and Claudia~A Sagastiz{\'a}bal.
\newblock {\em Numerical optimization: theoretical and practical aspects}.
\newblock Springer Science \& Business Media, 2006.

\bibitem{braun2023conditional}
Gábor Braun, Alejandro Carderera, Cyrille~W. Combettes, Hamed Hassani, Amin
  Karbasi, Aryan Mokhtari, and Sebastian Pokutta.
\newblock Conditional gradient methods, 2023.

\bibitem{combettes2018perspective}
Patrick~L Combettes.
\newblock Perspective functions: Properties, constructions, and examples.
\newblock {\em Set-Valued and Variational Analysis}, 26:247--264, 2018.

\bibitem{flamary2016optimal}
R{\'e}mi Flamary, Nicholas Courty, Davis Tuia, and Alain Rakotomamonjy.
\newblock Optimal transport for domain adaptation.
\newblock {\em IEEE Trans. Pattern Anal. Mach. Intell}, 1(1-40):2, 2016.

\bibitem{FW}
Marguerite Frank and Philip Wolfe.
\newblock An algorithm for quadratic programming.
\newblock {\em Naval Research Logistics Quarterly}, 3(1-2):95--110, 1956.

\bibitem{gidel2018frank}
Gauthier Gidel, Fabian Pedregosa, and Simon Lacoste-Julien.
\newblock Frank-{W}olfe splitting via augmented {L}agrangian method.
\newblock In {\em International Conference on Artificial Intelligence and
  Statistics}, pages 1456--1465. PMLR, 2018.

\bibitem{Jaggi}
Martin Jaggi.
\newblock Revisiting {F}rank-{W}olfe: Projection-free sparse convex
  optimization.
\newblock In {\em International conference on machine learning}, pages
  427--435. PMLR, 2013.

\bibitem{joulin2014efficient}
Armand Joulin, Kevin Tang, and Li~Fei-Fei.
\newblock Efficient image and video co-localization with {F}rank-{W}olfe
  algorithm.
\newblock In {\em Computer Vision--ECCV 2014: 13th European Conference, Zurich,
  Switzerland, September 6-12, 2014, Proceedings, Part VI 13}, pages 253--268.
  Springer, 2014.

\bibitem{Kiwiel}
Krzysztof~C. Kiwiel.
\newblock Proximal level bundle methods for convex nondifferentiable
  optimization, saddle-point problems and variational inequalities.
\newblock {\em Math. Programming}, 69(1):89--109, 1995.

\bibitem{kunisch2022fast}
Karl Kunisch and Daniel Walter.
\newblock On fast convergence rates for generalized conditional gradient
  methods with backtracking stepsize.
\newblock {\em Numerical Algebra, Control and Optimization}, pages 0--0, 2022.

\bibitem{lacoste2013block}
Simon Lacoste-Julien, Martin Jaggi, Mark Schmidt, and Patrick Pletscher.
\newblock Block-coordinate {F}rank-{W}olfe optimization for structural
  {S}{V}{M}s.
\newblock In {\em International Conference on Machine Learning}, pages 53--61.
  PMLR, 2013.

\bibitem{lavigne2023generalized}
Pierre Lavigne and Laurent Pfeiffer.
\newblock Generalized conditional gradient and learning in potential mean field
  games.
\newblock {\em Applied Mathematics \& Optimization}, 88(3):89, 2023.

\bibitem{lemarechal1995new}
Claude Lemar{\'e}chal, Arkadii Nemirovskii, and Yurii Nesterov.
\newblock New variants of bundle methods.
\newblock {\em Mathematical programming}, 69:111--147, 1995.

\bibitem{liu2023mean}
Kang Liu and Laurent Pfeiffer.
\newblock Mean field optimization problems: stability results and {L}agrangian
  discretization.
\newblock {\em arXiv preprint arXiv:2310.20037}, 2023.

\bibitem{pokutta2023frank}
Sebastian Pokutta.
\newblock The {F}rank-{W}olfe algorithm: a short introduction.
\newblock {\em Jah\-res\-be\-richt der Deut\-schen
  Ma\-the\-ma\-ti\-ker-Ver\-ei\-ni\-gung}, pages 1--33, 2023.

\bibitem{Silveti_Falls_2020}
Antonio Silveti-Falls, Cesare Molinari, and Jalal Fadili.
\newblock Generalized conditional gradient with augmented {L}agrangian for
  composite minimization.
\newblock {\em SIAM Journal on Optimization}, 30(4):2687–2725, January 2020.

\bibitem{wirth2023acceleration}
Elias Wirth, Thomas Kerdreux, and Sebastian Pokutta.
\newblock Acceleration of {F}rank-{W}olfe algorithms with open-loop step-sizes.
\newblock In {\em International Conference on Artificial Intelligence and
  Statistics}, pages 77--100. PMLR, 2023.

\bibitem{yurtsever2019conditional}
Alp Yurtsever, Olivier Fercoq, and Volkan Cevher.
\newblock A conditional-gradient-based augmented {L}agrangian framework.
\newblock In {\em International Conference on Machine Learning}, pages
  7272--7281. PMLR, 2019.

\bibitem{zualinescu}
Constantin Z{\u{a}}linescu.
\newblock {\em Convex analysis in general vector spaces}.
\newblock World Scientific Publishing Co., Inc., River Edge, NJ, 2002.

\bibitem{zhou2018limited}
Song Zhou, Swati Gupta, and Madeleine Udell.
\newblock Limited memory {K}elley's method converges for composite convex and
  submodular objectives.
\newblock {\em Advances in Neural Information Processing Systems}, 31, 2018.

\end{thebibliography}

\end{document}